\documentclass[10pt,a4paper,oneside]{amsart}

\usepackage{amsfonts,amssymb}
\usepackage{amsmath,amsthm} 
\usepackage{graphicx, color}
\usepackage{enumerate,array,enumitem}
\usepackage{url}
\usepackage{tikz-cd}
\usepackage[hidelinks]{hyperref}

\usepackage[all]{xy}

\usetikzlibrary{arrows.meta,calc}
\newtheorem{theorem}{Theorem}[section]

\newtheorem{proposition}[theorem]{Proposition}
\newtheorem{lemma}[theorem]{Lemma}
\newtheorem{corollary}[theorem] {Corollary}

\newtheorem*{claim*}{Claim}

\theoremstyle{remark}
\newtheorem{remark}[theorem]{Remark}

\newtheorem*{terminology}{Terminology}

\theoremstyle{definition}
\newtheorem{definition}[theorem]{Definition}
\newtheorem{notation}[theorem]{Notation}

\newcommand{\tn}[1]{\textnormal{#1}}
\newcommand{\lmod}[1]{#1\textnormal{-Mod}}

\newcommand{\Hom}{\textnormal{Hom}}
\newcommand{\rar}{\rightarrow}
\newcommand{\id}{\textnormal{id}}
\newcommand{\smsq}{{\mathbin{\vcenter{\hbox{\scalebox{0.45}{$\bigstar$}}}}}}
\newcommand{\vna}{\mathbf{vNa}}
\newcommand{\cvna}{\mathbf{CvNa}}

\newcommand{\cofu}[1]{\mathop{\boxtimes}_{#1}}
\newcommand{\tens}[1]{\mathop{*}\limits_{#1}}
\newcommand{\Ltwomap}{\mathfrak{s}}
\newcommand{\op}{\tn{op}}

\newcommand{\pull}[1]{#1^\smsq}
\newcommand{\push}[1]{#1_\smsq}

\newcommand{\vn}[1]{{#1}}
\newcommand{\res}[1]{#1^{\tikz[baseline=-1]{
\draw (0,0) -- (90:.07) (0,0) -- (162:.07) (0,0) -- (234:.07) (0,0) -- (306:.07)(0,0) -- (378:.07);
}}}
\newcommand{\vnres}[1]{\res{\vn{#1}}}
\newcommand{\ind}[1]{#1_{\tikz[baseline=-1]{
\draw (0,0) -- (90:.07) (0,0) -- (162:.07) (0,0) -- (234:.07) (0,0) -- (306:.07)(0,0) -- (378:.07);
}}}
\newcommand{\vnind}[1]{\ind{\vn{#1}}}

\newcommand{\ms}{(X,\mu)}

\newcommand{\vnmap}{\vn{f}\colon B \rar A}

\newcommand{\srmu}{\sqrt{\mu}}

\newcommand{\homtocond}[1]{#1^2}

\newcommand{\Wstarcat}{W^*\tn{-}\mathbf{Cat}}
\newcommand{\symWcat}{W^*\tn{-}\mathbf{SymTensCat}}

\makeatletter
\newcommand{\xRightarrow}[2][]{\ext@arrow 0359\Rightarrowfill@{#1}{#2}}
\makeatother

\title{A three-functor formalism for commutative von Neumann algebras}
\author{Andr\'e G. Henriques and Thomas A. Wasserman}   
\begin{document}

\begin{abstract} 
A three-functor formalism is the half of a six-functor formalism that supports the projection and base change formulas. In this paper, we provide a three-functor formalism for commutative von Neumann algebras and their modules. Using the Gelfand-Naimark theorem, this gives rise to a three-functor formalism for measure spaces and measurable bundles of Hilbert spaces. We use this to prove Fell absorption for unitary representations of measure groupoids.

The three-functor formalism for commutative von Neumann algebras takes values in $\mathrm{W}^*$-categories, and we discuss in what sense it is a unitary three-functor formalism.
\end{abstract}

%\address{}
%\email{}

\maketitle
\setcounter{tocdepth}{1}
\tableofcontents

\section{Introduction}

Let $\mathcal C$ be a `category of spaces' (such as the category of varieties over a field, or the category locally compact Hausdorff spaces), and let us assume that we are given for every object $X\in\mathcal C$ a `category of sheaves' on $X$ (typically some kind of derived category), denoted $\mathcal D_X$.
A \emph{six-functor formalism} assigns to every morphism $f:X\to Y$ in $\mathcal C$ a pair of adjoint functors $f^*:\mathcal D_Y\to \mathcal D_X$ and $f_*:\mathcal D_X\to \mathcal D_Y$, and to certain morphisms $f:X\to Y$ in $\mathcal C$ a pair of adjoint 
functors $f_!:\mathcal D_X\to \mathcal D_Y$ and $f^!:\mathcal D_Y\to \mathcal D_X$.
For every $X\in \mathcal C$, the category $\mathcal D_X$ is moreover equipped with a tensor product $\otimes$ that is preserved by each $f^*$, and has an inner-hom $\underline\Hom(-,-)$.
These functors satisfy compatibility conditions known as projection and base change.
The `six functors' of Grothendiek are $f^*$, $f_!$, $\otimes$, along with their right adjoints $f_*$, $f^!$, $\underline{\Hom}$.

When only $f^*$, $f_!$, $\otimes$ are present, this is called a  \emph{three-functor formalism}. 
The goal of this paper is to establish a three-functor formalism for the category of measure spaces and measurable maps (the opposite of the category of commutative von Neumann algebras).
The category $\mathcal D_X$ associated to a measure space $X$ is the category of measurable bundles of Hilbert spaces over $X$.
Equivalently, $\mathcal D_X$ is the category of representations of the commutative von Neumann algebra $L^\infty(X)$.

\subsection{Six-functor formalisms -- short historical overview.}
The formalism of six functors was originally introduced by Grothendieck and collaborators in the context of \'etale $\ell$-adic sheaves \cite{Grothendieck1977}, building on the development of \'etale cohomology \cite{Deligne1976}. By now it is recognised to exist in many contexts, among which: sheaves on locally compact Hausdorff spaces \cite{Leray1950,Borel1957,Deligne1976,Iversen1986, Kashiwara1990,Volpe2021}, $D$-modules \cite{Bernstein2007,Scholze2022}, quasi-coherent sheaves \cite{Lipman2009}, and motives \cite{Ayoub2007,Cisinski2019}.

Despite this ubiquity, the abstract notion of a six-functor formalism has only recently crystallised. A discussion of aspects of abstract six-functor formalisms first appears in \cite{Fausk2003,Joshua2003}. Formulations in terms of categories of correspondences were more recently discussed in \cite{Hoermann2018} and \cite{Gaitsgory2017}.
At last, general construction techniques of six-functor formalisms were studied and used in \cite{Liu2012,Mann2022}, and later expanded in \cite{Scholze2022,Heyer2024}.

As explained in \cite{Scholze2022}, a three-functor formalism is a symmetric monoidal functor out of a category of correspondences (or spans) in $\mathcal{C}$.\footnote{We will see later that our setup unfortunately doesn't allow for this particular way of packaging the definition. See Remark~\ref{rem 2.5}.}
This also captures the compatibility constraints between the three functors.
A six-functor formalism is a three-functor formalism in which the three functors $f^*$, $f_!$, $\otimes$ admit right adjoints.

\subsection{Measure spaces and commutative von Neumann algebras}

Here and below, a \emph{measure space} will refer to a pair $(X,[\mu])$, where $X$ is a set equipped with a $\sigma$-algebra, and $[\mu]$ is an equivalence class of measures on $X$ where two measures are deemed equivalent if they are absolutely continuous with respect to each other (they share the same null-sets).
All measure spaces will be assumed to be 
\emph{standard measure spaces}, meaning they have at most countably many atoms, and their non-atomic part is either empty, or isomorphic to $\mathbb{R}$.
A \emph{measurable map} $f:X\to Y$ between measure spaces is an equivalence class of measurable functions from $X$ to $Y$, where two functions $X\to Y$ are deemed equivalent if they agree outside of a measure zero set.

Given a measure space $X$, we write $L^\infty (X)$ for the associated commutative von Neumann algebra.
The Gelfand-Naimark theorem for measure spaces states that
the assignment $X\mapsto L^\infty(X)$ is an equivalence of categories between the category of standard measure spaces and measurable maps, and the opposite of the category of separable\footnote{A von Neumann algebra is called `separable' if it can be faithfully represented on a separable Hilbert space.} commutative von Neumann algebras and normal $*$-homomorphisms.

Unfortunately,
the correspondence provided by the Gelfand-Naimark theorem only works well in the presence of separability assumptions.\footnote{See \cite{Pavlov2022} for the current most general version of the Gelfand-Naimark theorem. The notion of measure space used in loc.~cit.~is however not well developed yet.}
In this paper, we construct our three-functor formalism purely in the language of commutative von Neumann algebras, which allows us to avoid any separability assumptions.

\subsection{Three-functor formalism for measure spaces}\label{sec:threeffmeasure}
We briefly describe our three-functor formalism from the point of view of measure spaces. (Later in this paper, we will be working purely in the language of commutative von Neumann algebras.)

Given a measure space $X$, 
the corresponding `category of sheaves' $\mathcal D_X$ is the category of measurable bundles of Hilbert spaces over $X$ (also called measurable fields of Hilbert spaces) and measurable bounded linear maps. 
This category is symmetric monoidal under the operation $\otimes_X$ of fiberwise tensor product of Hilbert spaces. 

The functor
\[
H\mapsto
L^2\big(X,H\otimes \Omega^{\frac{1}{2}}_X\big)
\]
which sends a measurable bundle $H$ over $X$ to the space of $L^2$ sections of $H\,{\otimes}\,\Omega^{\frac{1}{2}}_X$ provides an equivalence of categories between measurable\vspace{-.7mm} bundles over $X$ and normal representations of $L^\infty(X)$.
Here, $\Omega^{\frac{1}{2}}_X$ denotes the \emph{bundle of half-densities} on $X$, which we explain below.

Given a measure space $(X,[\mu])$, let $\underline{\mathbb C}_X$ denote the trivial line bundle over $X$.
For every $p\in \mathbb R$, there exists a canonical measurable line bundle $\Omega^p_X$ on $X$ called the bundle of $p$-densities. A measure $\mu$ on $X$ in the equivalence class $[\mu]$ provides a trivialisation of this bundle, denoted $(d\mu)^p: \underline{\mathbb C}_X \to \Omega^p_X$, and two such trivialisations are related by $(d\nu)^p = (d\mu)^p \circ (\tfrac{d\nu}{d\mu})^p$, where
$\tfrac{d\nu}{d\mu}$ is the Radon–Nikodym derivative.
For every $q\in \mathbb R$, the assignment $f\mapsto |f|^2$ defines a map from the space of sections of $\Omega_X^q$ to the space of sections of $\Omega_X^{2q}$.
Along with the canonical integration map $\int:\Gamma(X,\Omega^1)\to \mathbb C$,
this yields a measure-free description of the $L^2$-space of $X$:
\[
L^2(X) = \Big
\{f\in\Gamma\big(X,\Omega^{\frac{1}{2}}\big)\,:\,\int |f(x)|^2 <\infty\Big\}.
\]
If $H$ is a measurable bundle of Hilbert spaces over $X$, then the associated representation of $L^\infty(X)$ is given by the space of $L^2$ sections of 
$H\otimes \Omega^{\frac{1}{2}}$:
\[
L^2\big(X,H\otimes \Omega^{\frac{1}{2}}\big) := \Big\{f\in\Gamma\big(X,H\otimes \Omega^{\frac{1}{2}}\big)\,:\,\int \|f(x)\|^2 <\infty\Big\}.
\]

Given a measurable map $f\colon X\rar Y$, we now describe the associated pullback and pushforward functors. For reasons that will become apparent in the next section, we choose to denote them $\pull{f}$ and $\push{f}$ as opposed to the more standard $f^*$ and $f_!$.
The pullback functor $\pull{f}:\mathcal D_Y\to \mathcal D_X$ is easy to describe in the bundle perspective: it is given by pulling back measurable bundles over $Y$ along the measurable map~$f$.  This contravariantly corresponds to induction from $L^\infty(Y)$ to $L^\infty(X)$ (see \S\ref{sec:the three functors} for details of that construction). On the other hand, the pushforward  $\push{f}:\mathcal D_X\to \mathcal D_Y$ is easier to describe in the $L^\infty(X)$-module perspective: it corresponds to restriction of scalars along the homomorphism $L^\infty(Y)\to L^\infty(X)$.
On the side of measurable bundles, $\push{f}$ takes a measurable bundle $H$ on $X$ to the measurable bundle on $Y$
whose fiber over $y\in Y$ is given by
\[
L^2\big(X_y,H|_{X_y}\otimes \Omega_{X_y}^{\frac{1}{2}}\big).
\]
Here, $X_y:=f^{-1}(y)$ denotes the fiber of $X$ over $y$.

\subsection{Curiosities of our three-functor formalism}

The three-functor formalism 
constructed in this paper 
does not extend to a six-functor formalism in the sense of \cite{Scholze2022}.
Nevertheless, it behaves in many ways as a six-functor formalism: its covariant functors $\push{f}$ behave simultaneously like $f_*$ and $f_!$, and its contravariant functors $ \pull{f}$  behave simultaneously like $f^*$ and $f^!$. Let us explain:

A feature of six-functor formalisms which typically does not occur in a three-functor formalism is that each category $\mathcal D_X$ comes equipped with a (Verdier) duality functor
$\mathbb D_X:\mathcal D_X\to \mathcal D_X^{op}$
satisfying
\[
f_*(\mathbb D_X\mathcal F) \cong 
\mathbb D_Y(f_!\mathcal F)
\qquad
\text{and}
\qquad
f^!(\mathbb D_Y\mathcal G) \cong 
\mathbb D_X(f^*\mathcal G).
\]
Peculiarly, our three-functor formalism does support duality functors $\mathbb D$ as above, but with isomorphisms
\[
\push{f}(\mathbb D_X M) \cong 
\mathbb D_Y(\push{f} M)
\qquad
\text{and}
\qquad
\pull{f}(\mathbb D_YN) \cong 
\mathbb D_X(\pull{f}N).
\]
Another factoid which supports the claim that $\push{f}$ behaves like both $f_*$ and $f_!$, and that $ \pull{f}$  behaves like both $f^*$ and $f^!$, is that 
the functors $\push{f}$ and $\pull{f}$ are `adjoint' in the sense of \cite[Def~4.14]{Henriques2024}.
This is a different notion of adjoint which only makes sense for functor between $\mathrm{W}^*$-categories. It is ambidextrous, and does not agree with the usual notion of adjoint functor from category theory.

There is also a sense in which the functor $\pull{f}$ is `half-way' between $f^*$ and $f^!$, $\push{f}$ is `half-way' between $f_*$ and $f_!$, and the monoidal structure $\boxtimes$ on $\mathcal D_X$ is half-way between $\otimes$ and $\underline{\Hom}$ (and the Hilbert space direct sum is half-way between a product and a coproduct). One sees this as follows:

If $p:X\to *$ is the map to the terminal object then, in a usual six-functor formalism, the object $p_*p^*\mathbb I$ behaves like some kind of space of `functions on X'. The unit of the adjunction $p_* \dashv p^*$ then provides a map
$\mathbb I\to p_*p^*\mathbb I$ which we may interpret as the constant function 1.
Similarly, $p_!p^!\mathbb I$ behaves like a space of `densities on X', and the counit of the adjunction $p^! \dashv p_!$ provides a map $p_!p^!\mathbb I \to \mathbb I$
which we may interpret as integrating a density:
\begin{equation}
\label{eq: function vs densities}
\begin{split}
p_*p^*\mathbb I \,= \text{`functions on $X$'}\,\\
\,p_!\;\!p^!\;\!\mathbb I \,=\, \text{`densities on $X$'}.
\end{split}
\end{equation}
Our three-functor formalism supports neither 
$\mathbb I\to \push{p}\pull{p}\mathbb I$,
nor $\push{p}\pull{p}\mathbb I\to \mathbb I$.
Instead, $\push{p}\pull{p}\mathbb I=L^2(X)$ is the space of $L^2$ sections of the bundle of half-densities on the measure space $X$:
\[
\,\,\push{p}\pull{p}\mathbb I \,=\, \text{half-densities on }X,
\]
which is `exactly halfway' between the two statements in \eqref{eq: function vs densities}.

It is an interesting question whether there exist other setups that support similar `unitary $3$-functor formalisms'. It is reasonable to guess that the category of measurable groupoids supports a similar structure, though most likely not identical.

\subsection{Outline}
This paper is organised as follows. In Section~\ref{sec:threeff} we give the definition of a three-functor formalism. In Section~\ref{sec:commvn} we discuss preliminairies on commutative von Neumann algebras. We then set up the three-functor formalism in Section~\ref{sec4}. We end the paper by giving an application of the three-functor formalism of measure spaces to prove Fell absorption for representation of measure groupoids in Section~\ref{sec:fellabs}.

\subsubsection*{Acknowledgements}
We thank the anonymous referee for valuable feedback. TW acknowledges support from the European Commission through a Marie Sklodowska-Curie Individual Fellowship with reference number 101203556 -- DisQ FA, Richard Wade's Royal Society of Great Britain University Research Fellowship, and the University of Oxford.

\bigskip

{\small \noindent{\it Open access:} For the purpose of Open Access, the authors have applied a CC BY public copyright
licence to any Author Accepted Manuscript (AAM) version arising from this submission.}

\section{Three-functor formalisms}\label{sec:threeff}

A three-functor formalism is defined with respect to a reference ``category of spaces'' $\mathcal{C}$. 
For us, the relevant category of spaces $\mathcal{C}$ will be be the category of measure spaces and measurable maps. More precisely, it will be the opposite of the category of commutative von Neumann algebras.
In the later sections, everything will be phrased in the language of commutative von Neumann algebras.

\subsection{Definition of three-functor formalisms}
We begin by spelling out the extra structures that our category of spaces $\mathcal C$ needs to be equipped with:

\begin{definition}
    A \emph{category with exceptional morphisms and distinguished squares} is a category $\mathcal{C}$ equipped with:
\begin{itemize}
    \item A collection $\mathcal{E}$ of morphisms of $\mathcal{C}$ called \emph{exceptional morphisms}, which is closed under composition and contains all isomorphisms (so, in particular, $\mathcal{E}$ is a wide subcategory of $ \mathcal{C}$).
    \item
    A collection $\mathcal{S}$ of commuting squares
\begin{equation}\label{eq:distinguishedsquare}
	\begin{tikzcd}
		W \arrow[d,"\bar{g}"] \arrow[r,"\bar{f}"] &Y\arrow[d,"g"] \\
		X \arrow[r,"f"] & Z
	\end{tikzcd},
\end{equation}
with $g,\bar{g}\in \mathcal{E}$, 
called \emph{distinguished squares} which is closed under the operation of gluing two squares along an edge, and has the property that squares with $f=\bar{f}=\id $ or $g=\bar{g}=\id$ are in $\mathcal{S}$.
We also require that if $f\in\mathcal E$, then
the square obtained by swapping the roles of $X$ and $Y$ also be distinguished.
\end{itemize}
\end{definition}

\begin{remark}
All morphisms will be exceptional in our paper. We include the distinction in our definition to ensure that our definition of three-functor formalism accommodates examples from algebraic geometry.
\end{remark}

The instance of the above definition
that will be relevant to this paper is the category $\mathcal C$ of measure spaces and measurable maps
(or rather, the opposite of the category of commutative von Neumann algebras), with $\mathcal E=\mathcal C$,
and the distinguished squares given by the pullbacks of measure spaces.
The reader should be warned that pullbacks of measure spaces do not satisfy the usual universal property of pullbacks: they are therefore \emph{not} pullbacks in the usual sense of category theory.\medskip

\begin{definition}\label{def:threeffdata}
Let $\mathcal C$ be a category with exceptional morphisms and distinguished squares.
A \emph{three-functor formalism} on $\mathcal C$ consists of:
%A \emph{three-functor formalism} on $\mathcal C$ consists of:
	\begin{itemize} [wide=10pt, leftmargin=5mm]
		\item A functor $\mathcal{C}^\tn{op}\rar \mathbf{SymMonCat}$.\footnote{This is a really a pseudo-functor: we do not require an equality $\pull{g} \pull{f}= (fg)^\smsq$, but rather the data of a specified monoidal natural transformation $\pull{g} \pull{f}\cong (fg)^\smsq$, subject to its own coherences. See \cite{JohnsonYau2021} for more context about bicategories and pseudofunctors.}
We write $(\mathcal{D}_X,\otimes, \mathbb{I})$ for the value of that functor on an object $X\in \mathcal{C}$, and write $\pull{f}:\mathcal{D}_Y\to \mathcal{D}_X$
for the value of that functor on a morphism $f:X\to Y$. The symmetric monoidal functor $\pull{f}$ is called the \emph{pullback}.\medskip   
		\item A functor $\mathcal{E} \rar \mathbf{Cat}$ on the subcategory of exceptional morphisms which maps $X\in \mathcal{C}$ to (the underlying category of) $\mathcal{D}_X$. The value of that functor on a morphism $f:X\to Y$ is denoted by $\push{f}:\mathcal D_X\to \mathcal D_Y$, and is called the \emph{exceptional pushforward}.\medskip
	\end{itemize}
along with:
	\begin{itemize}[wide=10pt, leftmargin=5mm]
		\item (\emph{projection}) a natural isomorphism $$\push{f}(-)\otimes - \,\,\xRightarrow{\cong}\,\, \push{f}\left(- \otimes \pull{f}(-)\right)$$ for every morphism $f\in \mathcal E$.\medskip
		\item (\emph{base change}) a natural isomorphism $$\pull{f} \push{g} \,\,\xRightarrow{\cong}\,\, \push{\bar{g}}\pull{\bar{f}}$$ for every distinguished square~\eqref{eq:distinguishedsquare}.
	\end{itemize}

\begin{remark}
	The `three functors' are the functors $\pull{f}$, $\push{f}$, and $\otimes$.
\end{remark}

\noindent
This data is subject to the coherence conditions 1)--10) listed below. \medskip

1) Compatibility of projection formula with identity maps.\\
For $M,N\in \mathcal D_X$, we have:
\[
\begin{tikzcd}
    \push{\id}M\otimes N \arrow[r,"\cong"]\arrow[d,"\cong"] & M \otimes N\arrow[d,"\cong"]\\
    \push{\id}(M \otimes \pull{\id}N) \arrow[r,"\cong"]& \push{\id}(M\otimes N)
\end{tikzcd}.
\]

2) Compatibility of projection formula with composition of maps.\\
Given composable maps $X\stackrel g \to Y\stackrel f \to Z$ and objects $M\in\mathcal D_X$, $N\in\mathcal D_Z$, we have:
\[
\begin{tikzcd}
    \push{f}\push{g}M\otimes N \arrow[r,"\cong"]\arrow[d,"\cong"]&
    \push{f}\big(\push{g}M \otimes \pull{f}N\big) 
    \arrow[r,"\cong"]&
    \push{f}\push{g}\big(M\otimes \pull{g}\pull{f}N\big)\arrow[d,"\cong"]\\
    \push{(fg)}M\otimes N \arrow[rr,"\cong"]&&
    \push{(fg)}\big(M\otimes \pull{(fg)}N\big)
\end{tikzcd}.
\]

3) Compatibility of projection formula with identity objects.\\
For $f:X\to Y$ a map, and $M\in\mathcal D_X$, we have:
\[
\begin{tikzcd}
    \push{f}M\otimes \mathbb{I} \arrow[r,"\cong"]\arrow[d,"\cong"]& \push{f}M\arrow[d,"\cong"]\\
    \push{f}(M\otimes \pull{f}\mathbb{I})\arrow[r,"\cong"]& \push{f}(M \otimes \mathbb{I})
\end{tikzcd}.
\]

4) Compatibility of projection formula with tensor product.\\
Given a map $f:X\to Y$, and objects $M\in\mathcal D_X$ and  $N,P\in\mathcal D_Y$, we have:
\[
\begin{tikzcd}
    \push{f}M\otimes N \otimes P \arrow[r,"\cong"]\arrow[d,"\cong"]& \push{f}\big(M \otimes \pull{f}N \big)\otimes P \arrow[d,"\cong"]\\
    \push{f}\big(M \otimes \pull{f}(N\otimes P)\big) \arrow[r,"\cong"]& \push{f}\big(M \otimes \pull{f}N \otimes \pull{f}P \big)
\end{tikzcd}.
\]
where we have suppressed the associators from the diagram. 

5), 6) Compatibility of base change with identity maps:
\[
\begin{tikzcd}
    \pull{f} \push{\id} \arrow[r,"\cong"]\arrow[dr,"\cong"']&\push{\id}\pull{f}  \arrow[d,"\cong"]\\ & \pull{f}
\end{tikzcd}\qquad
\begin{tikzcd}
    \pull{\id} \push{f} \arrow[r,"\cong"]\arrow[dr,"\cong"']& \push{f}\pull{\id} \arrow[d,"\cong"]\\ & \push{f}
\end{tikzcd}.
\]

7) Compatibility of base change with horizontal composition of distinguished squares.
For every pair of horizontally composable distinguished squares
\[
\hspace{-1cm}
\begin{tikzcd}[baseline=0]
    \, \arrow[r,"\bar{f}_2"]\arrow[d,"\bar{\bar{g}}"]&\, \arrow[r,"\bar{f}_1"]\arrow[d,"\bar{g}"]& \, \arrow[d,"g"]\\
    \, \arrow[r,"f_2"] &\, \arrow[r,"f_1"] &\,
\end{tikzcd}
\,\,\mbox{ we have}\,\,\,\,\,\,\,\,\,\,
\begin{tikzcd}
\pull{f_2{}}\pull{f_1{}}\push{g}\arrow[d,"\cong"]\arrow[r,"\cong"]& \pull{f_2{}}\push{\bar{g}}\pull{\bar{f}_1{}} \arrow[r,"\cong"] & \push{\bar{\bar{g}}}\pull{\bar{f}_2{}} \pull{\bar{f}_1{}}\arrow[d,"\cong"]\\
    \pull{(f_1f_2)} \push{g} \arrow[rr,"\cong"]& & \push{\bar{\bar{g}}}\pull{(\bar{f}_1\bar{f}_2)}
\end{tikzcd}.
\hspace{-1.9cm}
\begin{tikzpicture}
\useasboundingbox (-2.5,0) rectangle +(-2.4,.1);
    \node[scale=0.75]{$\begin{matrix}
			\begin{tikzpicture}[baseline=-7pt]
				\draw[line width=0.7pt,arrows = {-Stealth[length=6pt, inset=5pt,width=6.5pt]}] (-0.35,-0.35)--(-0.6,-0.35);
				\draw[line width=0.7pt,arrows = {-Stealth[length=6pt, inset=5pt,width=6.5pt]}] (0,0)--(0,-0.3);
				\draw[line width=0.7pt,arrows = {-Stealth[length=6pt, inset=5pt,width=6.5pt]}] (-0.05,-0.35)--(-0.35,-0.35);
			\end{tikzpicture} &
				\begin{tikzpicture}[baseline=-7pt]
				\draw[line width=0.7pt,arrows = {-Stealth[length=6pt, inset=5pt,width=6.5pt]}] (-0.35,-0.35)--(-0.6,-0.35);
				\draw[line width=0.7pt,arrows = {-Stealth[length=6pt, inset=5pt,width=6.5pt]}] (0,0)--(-0.35,0);
				\draw[line width=0.7pt,arrows = {-Stealth[length=6pt, inset=5pt,width=6.5pt]}] (-0.35,0)--(-0.35,-0.35);
			\end{tikzpicture}
			& \begin{tikzpicture}[baseline=-7pt]
				\draw[line width=0.7pt,arrows = {-Stealth[length=6pt, inset=5pt,width=6.5pt]}] (-0.6,0)--(-0.6,-0.35);
				\draw[line width=0.7pt,arrows = {-Stealth[length=6pt, inset=5pt,width=6.5pt]}] (0,0)--(-0.35,0);
				\draw[line width=0.7pt,arrows = {-Stealth[length=6pt, inset=5pt,width=6.5pt]}] (-0.35,0)--(-0.6,0);
			\end{tikzpicture}\\	
			\begin{tikzpicture}[baseline=-7pt]
				\draw[line width=0.7pt,arrows = {-Stealth[length=6pt, inset=5pt,width=6.5pt]}] (0,0)--(0,-0.3);
				\draw[line width=0.7pt,arrows = {-Stealth[length=6pt, inset=5pt,width=6.5pt]}] (-0.05,-0.35)--(-0.6,-0.35);
			\end{tikzpicture} &  &
			\begin{tikzpicture}[baseline=-7pt]
				\draw[line width=0.7pt,arrows = {-Stealth[length=6pt, inset=5pt,width=6.5pt]}] (0,0)--(-0.6,0);
				\draw[line width=0.7pt,arrows = {-Stealth[length=6pt, inset=5pt,width=6.5pt]}] (-0.6,0)--(-0.6,-0.35);
			\end{tikzpicture}
    \end{matrix}$};\end{tikzpicture}
\]

8) Compatibility of base change with vertical composition of distinguished squares.
For every pair of vertically composable distinguished squares
\[
\hspace{-.2cm}
\begin{tikzcd}[baseline=0]
  \,  \arrow[r,"\bar{\bar{f}}"']\arrow[d,"\bar{g}_2"]&\arrow[d,"g_2"]\,\\ \,\arrow[r,"\bar{f}"'] \arrow[d,"\bar{g}_1"] & \, \arrow[d,"g_1"]\\ \, \arrow[r,"f"']& \,
\end{tikzcd}\!\!\!\!
\,\,\,\,\,\,\,\,\,\,\,\,\mbox{ we have}\,\,\,\,\,\,\,\,\,\,
\begin{tikzcd}
\pull{f}\push{g_1{}}\push{g_2{}}\arrow[r,"\cong"]\arrow[d,"\cong"]& \push{\bar{g}_1{}}\pull{\bar{f}}\push{g_2{}}\arrow[r,"\cong"] &\push{\bar{g}_1{}}\push{\bar{g}_2{}} \pull{\bar{\bar{f}}} \arrow[d,"\cong"]\\
\pull{f}\push{(g_1g_2)} \arrow[rr,"\cong"] & & \push{(\bar{g}_1\bar{g}_2)}\pull{\bar{\bar{f}}}
\end{tikzcd}.
\hspace{-1.9cm}
\begin{tikzpicture}
\useasboundingbox (-2.5,0) rectangle +(-2.4,.1);
    \node[scale=0.75]{$\begin{matrix}
			\begin{tikzpicture}[baseline=-7pt]
				\draw[line width=0.7pt,arrows = {-Stealth[length=6pt, inset=5pt,width=6.5pt]}] (0,0.3)--(0,0);
				\draw[line width=0.7pt,arrows = {-Stealth[length=6pt, inset=5pt,width=6.5pt]}] (0,0)--(0,-0.3);
				\draw[line width=0.7pt,arrows = {-Stealth[length=6pt, inset=5pt,width=6.5pt]}] (-0.05,-0.35)--(-0.35,-0.35);
			\end{tikzpicture} &
				\begin{tikzpicture}[baseline=-7pt]
				\draw[line width=0.7pt,arrows = {-Stealth[length=6pt, inset=5pt,width=6.5pt]}] (0,0.3)--(0,0);
				\draw[line width=0.7pt,arrows = {-Stealth[length=6pt, inset=5pt,width=6.5pt]}] (0,0)--(-0.35,0);
				\draw[line width=0.7pt,arrows = {-Stealth[length=6pt, inset=5pt,width=6.5pt]}] (-0.35,0)--(-0.35,-0.35);
			\end{tikzpicture}
			& 	\begin{tikzpicture}[baseline=-7pt]
				\draw[line width=0.7pt,arrows = {-Stealth[length=6pt, inset=5pt,width=6.5pt]}] (0,0.3)--(-0.35,0.3);
				\draw[line width=0.7pt,arrows = {-Stealth[length=6pt, inset=5pt,width=6.5pt]}] (-0.35,0.3)--(-0.35,0);
				\draw[line width=0.7pt,arrows = {-Stealth[length=6pt, inset=5pt,width=6.5pt]}] (-0.35,0)--(-0.35,-0.3);
			\end{tikzpicture} \\	
			\begin{tikzpicture}[baseline=-7pt]
			\draw[line width=0.7pt,arrows = {-Stealth[length=6pt, inset=5pt,width=6.5pt]}] (0,0.3)--(0,-0.3);
			\draw[line width=0.7pt,arrows = {-Stealth[length=6pt, inset=5pt,width=6.5pt]}] (-0.05,-0.35)--(-0.35,-0.35);
		\end{tikzpicture}&  &
				\begin{tikzpicture}[baseline=-7pt]
				\draw[line width=0.7pt,arrows = {-Stealth[length=6pt, inset=5pt,width=6.5pt]}] (0,0.3)--(-0.35,0.3);
				\draw[line width=0.7pt,arrows = {-Stealth[length=6pt, inset=5pt,width=6.5pt]}] (-0.35,0.3)--(-0.35,-0.3);
			\end{tikzpicture} 
\end{matrix}$};\end{tikzpicture}
\]

9) Compatibility between the projection formula and the base change. For every distinguished square \eqref{eq:distinguishedsquare},
and objects $M\in\mathcal D_Y$ and  $N \in \mathcal D_Z$, we have:

\begin{equation}\label{eq:bcandprojcohone}
	\begin{tikzcd}[column sep=small]
	\pull{f}\big(\push{g}M\otimes N\big)\arrow[r,"\cong"]\arrow[d,"\cong"]&\arrow[r,"\cong"]	\pull{f} \push{g}M\otimes \pull{f}N&\arrow[d,"\cong"]	\push{\bar{g}}\pull{\bar{f}}M\otimes \pull{f}N\\
	\pull{f} \push{g}(M\otimes \pull{g}N)\arrow[d,"\cong"]& &\arrow[d,"\cong"]\push{\bar{g}}\big(\pull{\bar{f}}M\otimes \pull{\bar{g}} \pull{f}N\big)\\ 
	\push{\bar{g}}\pull{\bar{f}}\big(M\otimes \pull{g}N\big)\arrow[r,"\cong"]& \push{\bar{g}}\big(\pull{\bar{f}}M\otimes \pull{\bar{f}} \pull{g}N\big)\arrow[r,"\cong"]&\push{\bar{g}}\big(\pull{\bar{f}}M\otimes (g\bar{f})^\smsq N\big)
	\end{tikzcd}
\hspace{-2.2cm}
\begin{tikzpicture}
\useasboundingbox (-2.81,0) rectangle +(-2.8,.01);
    \node[scale=.5]{$
        \begin{matrix}
			\begin{tikzpicture}[baseline=-7pt]
				\filldraw (0,-0.6) circle (0.05);
				\draw[line width=0.7pt,arrows = {-Stealth[length=6pt, inset=5pt,width=6.5pt]}] (0,0)--(0,-0.5);
				\draw[line width=0.7pt,arrows = {-Stealth[length=6pt, inset=5pt,width=6.5pt]}] (-0.1,-0.55)--(-0.6,-0.55);
			\end{tikzpicture}&%\hspace{0.7cm}
			\begin{tikzpicture}[baseline=-7pt]
				\draw[line width=0.7pt,arrows = {-Stealth[length=6pt, inset=5pt,width=6.5pt]}] (0,0)--(0,-0.5);
				\draw[line width=0.7pt,arrows = {-Stealth[length=6pt, inset=5pt,width=6.5pt]}] (0,-0.5)--(-0.6,-0.5);
				\draw[line width=0.7pt,arrows = {-Stealth[length=6pt, inset=5pt,width=6.5pt]}] (0.05,-0.63)--(-0.6,-0.63);
			\end{tikzpicture}&
			\begin{tikzpicture}[baseline=-7pt]
				\draw[line width=0.7pt,arrows = {-Stealth[length=6pt, inset=5pt,width=6.5pt]}] (0,0)--(-0.6,0);
				\draw[line width=0.7pt,arrows = {-Stealth[length=6pt, inset=5pt,width=6.5pt]}] (-0.6,0)--(-0.6,-0.5);
				\draw[line width=0.7pt,arrows = {-Stealth[length=6pt, inset=5pt,width=6.5pt]}] (0.05,-0.63)--(-0.6,-0.63);
			\end{tikzpicture}\\
			\begin{tikzpicture}[baseline=-7pt]
				\filldraw (0,0) circle (0.05);
				\draw[line width=0.7pt,arrows = {-Stealth[length=6pt, inset=5pt,width=6.5pt]}] (0.1,-0.55)--(0.1,0)--(-0.1,0)--(-0.1,-0.55);
				\draw[line width=0.7pt,arrows = {-Stealth[length=6pt, inset=5pt,width=6.5pt]}]  (-0.1,-0.55)--(-0.6,-0.55);
			\end{tikzpicture}&&
			\begin{tikzpicture}[baseline=-7pt]
				\draw[line width=0.7pt,arrows = {-Stealth[length=6pt, inset=5pt,width=6.5pt]}] (0,0)--(-0.6,0);
				\draw[line width=0.7pt,arrows = {-Stealth[length=6pt, inset=5pt,width=6.5pt]}] (0.05,-0.6)--(-0.6,-0.6);
				\draw[line width=0.7pt,arrows = {-Stealth[length=6pt, inset=5pt,width=6.5pt]}] (-0.6,-0.6)--(-0.6,0)--(-0.8,0)--(-0.8,-0.65);
			\end{tikzpicture}\\
			\begin{tikzpicture}[baseline=-7pt]
				\filldraw (0,0.65) circle (0.05);
				\draw[line width=0.7pt,arrows = {-Stealth[length=6pt, inset=5pt,width=6.5pt]}] (0,0)--(0,0.6);
				\draw[line width=0.7pt,arrows = {-Stealth[length=6pt, inset=5pt,width=6.5pt]}] (-0.05,0.6)--(-0.6,0.6)--(-0.6,0);
			\end{tikzpicture}&%\hspace{1.7cm}
			\begin{tikzpicture}[baseline=-7pt]
				\draw[line width=0.7pt,arrows = {-Stealth[length=6pt, inset=5pt,width=6.5pt]}] (0,0)--(0,0.55);
				\draw[line width=0.7pt,arrows = {-Stealth[length=6pt, inset=5pt,width=6.5pt]}] (0,0.55)--(-0.6,0.55);
				\draw[line width=0.7pt,arrows = {-Stealth[length=6pt, inset=5pt,width=6.5pt]}] (0,0.65)--(-0.6,0.65);
				\draw[line width=0.7pt,arrows = {-Stealth[length=6pt, inset=5pt,width=6.5pt]}] (-0.65,0.5)--(-0.65,0);
			\end{tikzpicture}&
			\begin{tikzpicture}[baseline=-7pt]
				\draw[line width=0.7pt,arrows = {-Stealth[length=6pt, inset=5pt,width=6.5pt]}] (0,0)--(0,0.55)--(-0.6,0.55);
				\draw[line width=0.7pt,arrows = {-Stealth[length=6pt, inset=5pt,width=6.5pt]}] (0,0.65)--(-0.6,0.65);
				\draw[line width=0.7pt,arrows = {-Stealth[length=6pt, inset=5pt,width=6.5pt]}] (-0.65,0.5)--(-0.65,0);
			\end{tikzpicture},
	\end{matrix}$};
\end{tikzpicture}
\end{equation}
using that $g\bar{f}=f\bar{g}$.

10) Compatibility of the horizontal and vertical projection formulas inside a distinguished square. 
If all the arrows of a distinguished square \eqref{eq:distinguishedsquare} are in $\mathcal E$, then for all $M\in\mathcal D_X$ and  $N \in \mathcal D_Y$,
we have:
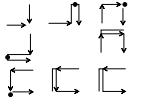
\begin{equation}\label{eq:bcandprojcohtwo}
	\begin{tikzcd}[column sep=small]
	\push{f}M\otimes \push{g}N\arrow[r,"\cong"]\arrow[d,"\cong"]& 	\push{g}\big(\pull{g} \push{f}M\otimes N\big)\arrow[r,"\cong"]& 	\push{g}\big(\push{\bar{f}}\pull{\bar{g}}M\otimes N\big)\arrow[d,"\cong"]\\
		\push{f} \big(M\otimes \pull{f} \push{g}N\big)\arrow[d,"\cong"]& &	\push{g}\push{\bar{f}}\big(\pull{\bar{g}}M\otimes  \pull{\bar{f}}N\big)\arrow[d,"\cong"]\\ 
	\big(M\otimes \push{\bar{g}} \pull{\bar{f}}N\big)\arrow[r,"\cong"]&\push{f}\push{\bar{g}}\big(\pull{\bar{g}}M\otimes \pull{\bar{f}}N\big)\arrow[r,"\cong"]&(f\bar{g})_\smsq\big(\pull{\bar{g}}M\otimes \pull{\bar{f}}N\big)
	\end{tikzcd}.
\hspace{-2.2cm}
\begin{tikzpicture}[baseline=0pt]
\useasboundingbox (-2.81,0) rectangle +(-2.8,.01);
    \node[scale=0.5]{$
		\begin{matrix}
		\begin{tikzpicture}[baseline=-7pt]
				\draw[line width=0.7pt,arrows = {-Stealth[length=6pt, inset=5pt,width=6.5pt]}] (0,0)--(0,-0.5);
				\draw[line width=0.7pt,arrows = {-Stealth[length=6pt, inset=5pt,width=6.5pt]}] (-0.6,-0.55)--(-0.1,-0.55);
			\end{tikzpicture}&%\hspace{1.7cm}
			\begin{tikzpicture}[baseline=-7pt]
				\filldraw (0.1,0) circle (0.05);
				\draw[line width=0.7pt,arrows = {-Stealth[length=6pt, inset=5pt,width=6.5pt]}] (0,-0.5)--(0,0)--(0.2,0)--(0.2,-0.55);
				\draw[line width=0.7pt,arrows = {-Stealth[length=6pt, inset=5pt,width=6.5pt]}] (-0.6,-0.5)--(0,-0.5);
			\end{tikzpicture}&
			\begin{tikzpicture}[baseline=-7pt]
				\filldraw (0.1,0) circle (0.05);
				\draw[line width=0.7pt,arrows = {-Stealth[length=6pt, inset=5pt,width=6.5pt]}] (-0.5,0)--(0,0);
				\draw[line width=0.7pt,arrows = {-Stealth[length=6pt, inset=5pt,width=6.5pt]}] (0.05,-0.1)--(0.05,-0.55);
				\draw[line width=0.7pt,arrows = {-Stealth[length=6pt, inset=5pt,width=6.5pt]}] (-0.5,-0.5)--(-0.5,0);
			\end{tikzpicture}\\
			\begin{tikzpicture}[baseline=-7pt]
				\filldraw (-0.6,-0.625) circle (0.05);
				\draw[line width=0.7pt,arrows = {-Stealth[length=6pt, inset=5pt,width=6.5pt]}] (0,0)--(0,-0.55);
				\draw[line width=0.7pt,arrows = {-Stealth[length=6pt, inset=5pt,width=6.5pt]}] (0,-0.55)--(-0.6,-0.55)--(-0.6,-0.7)--(0,-0.7);
			\end{tikzpicture}&&
			\begin{tikzpicture}[baseline=-7pt]
				\draw[line width=0.7pt,arrows = {-Stealth[length=6pt, inset=5pt,width=6.5pt]}] (0.1,0.1)--(-0.5,0.1)--(-0.5,0)--(0.1,0);
				\draw[line width=0.7pt,arrows = {-Stealth[length=6pt, inset=5pt,width=6.5pt]}] (0.1,0)--(0.1,-0.5);
				\draw[line width=0.7pt,arrows = {-Stealth[length=6pt, inset=5pt,width=6.5pt]}] (-0.5,-0.5)--(-0.5,0);
			\end{tikzpicture}\\
			\begin{tikzpicture}[baseline=-20pt]
				\filldraw (-0.6,-0.65) circle (0.05);
				\draw[line width=0.7pt,arrows = {-Stealth[length=6pt, inset=5pt,width=6.5pt]}] (0,0)--(-0.6,0);
				\draw[line width=0.7pt,arrows = {-Stealth[length=6pt, inset=5pt,width=6.5pt]}] (-0.6,0)--(-0.6,-0.55);
				\draw[line width=0.7pt,arrows = {-Stealth[length=6pt, inset=5pt,width=6.5pt]}]  (-0.55,-0.575)--(0,-0.575);
			\end{tikzpicture}&%\hspace{1.7cm}
			\begin{tikzpicture}[baseline=-4pt]
				\draw[line width=0.7pt,arrows = {-Stealth[length=6pt, inset=5pt,width=6.5pt]}] (-0.6,0)--(0,0);
				\draw[line width=0.7pt,arrows = {-Stealth[length=6pt, inset=5pt,width=6.5pt]}] (0,0.6)--(-0.6,0.6);
				\draw[line width=0.7pt,arrows = {-Stealth[length=6pt, inset=5pt,width=6.5pt]}] (-0.7,0)--(-0.7,0.6)--(-0.6,0.6)--(-0.6,0);
			\end{tikzpicture}&
			\begin{tikzpicture}[baseline=-4pt]
				\draw[line width=0.7pt,arrows = {-Stealth[length=6pt, inset=5pt,width=6.5pt]}] (0,0.6)--(-0.6,0.6);
				\draw[line width=0.7pt,arrows = {-Stealth[length=6pt, inset=5pt,width=6.5pt]}] (-0.7,0)--(-0.7,0.6)--(-0.6,0.6)--(-0.6,0)--(0,0);
			\end{tikzpicture}
	\end{matrix}$};
\end{tikzpicture}
\end{equation}

\noindent
This finishes the definition of a three-functor formalism.
\end{definition}

In the case that is relevant for the present paper, there will be no distinction between morphisms and exceptional morphisms.
Correspondingly, we will just write `pushforward' instead of `exceptional pushforward'.

\begin{remark}\label{rem 2.5}
	In categorically better behaved situations, there is a more concise definition of a three-functor formalism as a functor on a category of correspondences, see \cite{Scholze2022}. In our setup, this approach is however not applicable as it relies for example on the diagonal map $X\rar X\times X$ to produce the monoidal structure on $\mathcal D_X$. But the product of measure spaces (which is not a categorical product) does not admit a diagonal.\footnote{For example, the diagonal map $\mathbb R\to \mathbb R^2$ is not measurable, as the preimage of a Lebesgue-measurable subset of $\mathbb R^2$ can fail to be measurable in $\mathbb R$.}
\end{remark}

\subsection{Variants and enhancements of three-functor formalisms}
\label{sec:Variants and enhancements of three-functor formalisms}
The example of a three-functor formalism studied in this paper (measure spaces and measurable maps) admits extra structure not yet included in the above definition: each category $\mathcal D_X$ comes equipped with a contravariant involution $\mathbb D_X$. This should be thought of as similar to 
the map $\mathcal F \mapsto \mathbb D_X\mathcal F:=\underline{\Hom}(\mathcal F,\omega_X)$ which exists in any \emph{six}-functor formalism (here, the dualising object $\omega_X\in\mathcal D_X$ is given by $\omega_X=p^!\mathbb I$, for $p:X\to *$, and $\mathcal F\in\mathcal D_X$ is any object). 

In a six-functor formalism, we always have natural isomorphisms (see for example \cite[Section A.5]{Cisinski2019})
\[
f_*(\mathbb D_X\mathcal F) \cong 
\mathbb D_Y(f_!\mathcal F)
\qquad
\text{and}
\qquad
f^!(\mathbb D_Y\mathcal G) \cong 
\mathbb D_X(f^*\mathcal G).
\]
Our three-functor formalism has the interesting property that it supports an involutive duality functor which comes with isomorphisms
\[
\push{f}(\mathbb D_X M) \cong 
\mathbb D_Y(\push{f} M)
\qquad
\text{and}
\qquad
\pull{f}(\mathbb D_YN) \cong 
\mathbb D_X(\pull{f}N).
\]
In that sense, the functor $\pull{f}$ plays the role of both $f^*$ and $f^!$ (which is why we didn't denote it $f^*$), and similarly $\push{f}$ plays the role of both $f_*$ and $f_!$.
We call the resulting structure an \emph{involutive three-functor formalism}:

\begin{definition}\label{def:inv3ff}
    An \emph{involutive three-functor formalism} is a three-functor formalism $(\mathcal{C},\mathcal{E},\mathcal{S}, \pull{f}, \push{f}, \otimes)$ together with an involutive \emph{duality functor}
\[
\mathbb D_X\colon \mathcal{D}_X \longrightarrow\mathcal{D}_X^{op}
\]
    for every $X\in \mathcal{C}$ (when no confusion can arise we will write $\mathbb{D}$ in place of $\mathbb{D}_X$).
The duality functor is involutive in the sense that it comes equipped with an invertible natural transformation $\varphi_M:M \to \mathbb D^2 M$ for all $M\in\mathcal D_X$, satisfying $ \varphi_{\mathbb D M}=\mathbb D(\varphi_M)^{-1}$.
It is also compatible with the tensor structure on $\mathcal D_X$, meaning that there are monoidally coherent natural isomorphisms $r:\mathbb I\to \mathbb D \mathbb I$ and $\nu_{M,N}:\mathbb DM\,\otimes\, \mathbb DN \to \mathbb D(M\otimes N)$ satisfying $\varphi_{\mathbb I}=\mathbb D(r)^{-1} \circ r$
and $\varphi_{M\otimes N} =
\mathbb D(\nu_{M,N})^{-1}
\circ
\nu_{\mathbb DM,\mathbb DN}
\circ (\varphi_M\otimes \varphi_N)$.

We additionally require that for every morphism $f\colon X \rar Y$ in $\mathcal{C}$ we have natural isomorphisms
\begin{equation}
\label{eq: map xi}
\begin{split}
\xi_M \,&\colon \push{f}(\mathbb{D}_XM) \xrightarrow{\,\,\cong\,\,} \mathbb{D}_Y\push{f}(M)
\\
\zeta_N \,&\colon \pull{f}(\mathbb{D}_YN) \xrightarrow{\,\,\cong\,\,} \mathbb{D}_X\pull{f}(N)
\end{split}
\end{equation}
satisfying
$$\xi_{\mathbb{D}_XM}=\mathbb{D}_Y(\xi_{M})\circ \varphi_{\push{f}(M)}\circ 
\push{f}(\varphi_{M})^{-1}
$$ for every $M\in \mathcal{D}_X$, and 
$$\zeta_{\mathbb{D}_YN}=\mathbb{D}_X(\zeta_{N})\circ \varphi_{\push{f}(N)}\circ 
\push{f}(\varphi_{N})^{-1}
$$
for every $N\in \mathcal{D}_Y$.
The map $\zeta$ is furthermore required to be monoidal in the sense that %\vspace{-.7mm}
$$\zeta_{{\mathbb I}_Y}= \mathbb D(i)\circ r_X \circ i \circ \pull{f}(r_Y)^{-1}$$ where $i$ is the isomorphism $i \colon \pull{f}(\mathbb{I}_Y)\xrightarrow{\cong} {\mathbb I}_X$, and $$\zeta_{M \otimes N}= \mathbb{D}_X(\mu_{M,N })\circ \nu_{\pull{f}M,\pull{f}N}\circ (\zeta_M \otimes \zeta_N)\circ\mu_{\mathbb{D}_YM,\mathbb{D}_YN}\circ\pull{f}(\nu_{M,N})^{-1}$$
where $\mu$ is the isomorphism $$\mu_{M,N}:\pull{f}(M\otimes N) \to \pull{f}(M)\otimes \pull{f}(N).$$

Finally, the duality needs to be compatible with projection and base change in the sense that the following two diagrams commute:
\[
\begin{tikzpicture}[yscale=1.1]
\node (A) at (0,2) {$\push{f} (\mathbb D_X M) \otimes \mathbb D_Y N$};
\node (B) at (4,2) {$\push{f} (\mathbb D_X M \otimes \pull{f}(\mathbb D_Y N))$};
\node (C) at (8,2) {$\push{f} (\mathbb D_X M \otimes \mathbb D_X(\pull{f} N))$};
\node (D) at (8,1) {$\push{f} (\mathbb D_X (M \otimes \pull{f} N))$};
\node (E) at (8,0) {$\mathbb D_Y(\push{f} (M \otimes \pull{f} N))$};
\node (F) at (0,0) {$\mathbb D_Y(\push{f} M \otimes N)$};
\node (G) at (0,1) {$\mathbb D_Y(\push{f} M) \otimes \mathbb D_YN$};
\draw[->] (A) -- (B);
\draw[->] (B) -- (C);
\draw[->] (C) -- (D);
\draw[->] (D) -- (E);
\draw[->] (F) -- (E);
\draw[->] (G) -- (F);
\draw[->] (A) -- (G);
\end{tikzpicture}
\]
for $f\colon X \rar Y $, $M\in \mathcal{D}_X$, $N\in \mathcal{D}_Y$, and
\[
\begin{tikzcd}
    \pull{g} \push{f} \mathbb{D}_X \arrow[r]\arrow[d]& \push{\bar{f}}\pull{\bar{g}}\mathbb{D}_X\arrow[r]&   \push{\bar{f}}\mathbb{D}_{W} \pull{\bar{g}}\arrow[d]\\
    \pull{g}\mathbb{D}_{Z}\push{f}\arrow[r] &
    \mathbb{D}_Y \pull{g}\push{f}\arrow[r]&
    \mathbb{D}_Y \push{\bar f}\pull{\bar g}
\end{tikzcd}
\]
where $f$, $g$, $\bar f$, $\bar g$ are as in \eqref{eq:distinguishedsquare}.
This finishes Definition~\ref{def:inv3ff}.
\end{definition}

Recall that a category $\mathcal D$ is called a \emph{dagger category} if it comes equipped with maps $\dagger:\Hom_{\mathcal D}(X,Y)\to \Hom_{\mathcal D}(Y,X)$ that are compatible with composition of morphisms, and satisfy $f^{\dagger\dagger}=f$. If the category is $\mathbb C$-linear, we typically also require the operation $f\mapsto f^\dagger$ be be antilinear.
An isomorphism $u:X\to Y$ in a dagger category is called \emph{unitary} if $u^{-1}=u^\dagger$.

If the categories $\mathcal D_X$ that appear in a three-functor formalism are dagger categories (which they are in our example of interest), and if the functors $\pull{f}$, $\push{f}$ and $-\otimes M$ are all dagger functors, then we call our involutive three-functor formalism a bi-involutive three-functor formalism:\footnote{See \cite[Section 4.1]{Henriques2023} for more about (bi-)involutive categories and functors between them.}

\begin{definition}\label{def: biinv 3ff}
A \emph{bi-involutive three-functor formalism} is an involutive three-functor formalism where each category $\mathcal D_X$ is a dagger category, the functors $\pull{f}$, $\push{f}$, $\otimes$, $\mathbb D$ are dagger functors, and all the natural isomorphisms which appear in the various layers of the definition are unitary.
\end{definition}

In a {six}-functor formalism the functor $f^*$, $f_!$ and $M\otimes -$ have right adjoints. The three-functor formalism studied in this paper does not extend to a six-functor formalism.
    Nevertheless, $\pull{f}$, $\push{f}$ and $M\otimes -$ do have `adjoints' in the sense of \cite[Def~4.14]{Henriques2024}: this is a notion of adjoint functor that only makes sense for $\mathrm{W}^*$-categories, and which is not the usual notion from category theory.
    (In that theory, there is no distinction between left and right adjoints.)    
With this different notion of adjoint, the functors $\pull{f}$ and $\push{f}$ are each other's adjoints, and $-\otimes M$ is adjoint to $-\otimes \mathbb{D} M$.
This observation adds weight to the argument that $\pull{f}$ plays the role of both $f^*$ and $f^!$ , and that $\push{f}$ plays the role of both $f_*$ and $f_!$.
The role of $\underline{\Hom}(-,-)$ is then played by $\mathbb{D} (-)\otimes -$.

\begin{remark}
If one defines 
$\underline{\Hom}(M,N):=\mathbb{D}M\otimes N$, and $\omega_X:=\pull{p}\mathbb I$ for $p:X\to *$ (i.e.~$\omega_X=\mathbb I$), then the formula 
$\mathbb D_XM =\underline{\Hom}_X(M,\omega_X)$ holds true.
\end{remark}

The overall structure that our three-functor formalism possesses is what we call a \emph{unitary three-functor formalism}. It is a bi-involutive three-functor formalism with values in $\mathrm{W}^*$-categories, such that $(\pull{f},\push{f})$ and $(-\otimes M,-\otimes \mathbb{D} M)$ are adjoint pairs in the sense of \cite{Henriques2024}.

\section{Commutative von Neumann algebras and their modules}\label{sec:commvn}

Our three-functor formalism has the opposite of the category of commutative von Neumann algebras as its underlying category of spaces. It assigns to a von Neumann algebra $A$ its category of modules, and the three functors are induction, restriction, and fusion (the latter also known as \emph{Connes fusion} %\cite{Sauvageot1985} 
\cite{Wassermann1995}).

\subsection{The category of modules}
If $A$ is a commutative von Neumann algebra,
the category %$\lmod{A}$
of $A$-modules (Hilbert spaces equipped with normal unital $A$-actions) and bounded $A$-linear maps  is a complete $\mathrm{W}^*$-tensor category in the sense of \cite{Henriques2024}. Its tensor product is given by the operation $-\boxtimes_A-$ of \emph{fusion} over $A$, which we now describe (standard references include \cite{Sauvageot1983,Connes1994}).

Let $A$ be a von Neumann algebra, not necessarily commutative, 
let $M$ be a right $A$-module, and let $N$ be a left $A$-module. The fusion $M\cofu{A} N$ of $M$ and $N$ is the Hilbert space completion of $\Hom_A(L^2 A, M)\otimes_A N$
with respect to the inner product
\begin{equation}\label{e:innerprodonfusion}
    \langle
\phi_1\otimes n_1,\phi_2\otimes n_2
\rangle:=\langle (\phi_2^*\phi_1)
n_1,n_2
\rangle,
\end{equation} 
where $\phi_2^*\phi_1:L^2A\to L^2A$ is identified with an element of $A$, and $L^2A$ is Haagerup's standard form of the von Neumann algebra $A$ \cite{Haagerup}.
Fusion can be alternatively described as a completion of $\Hom_A(L^2 A, M)\otimes_A L^2A\otimes_A \Hom_A(L^2A,N)$, or a completion of $M\otimes_A \Hom_A(L^2A,N)$. 
The operation of fusion over $A$ is associative and unital, with unit provided by $L^2A$.

In the later sections,
we will make heavy usage of the following result:% from \cite{Henriques2024}:
\begin{lemma} [{\cite[Cor 3.25]{Henriques2024}}]\label{lem:mainlemma}
	Let $A$ and $B$ be von Neumann algebras. Then
	\begin{enumerate}
		\item a functor $F\colon \lmod{A}\rightarrow \lmod{B}$ is determined by and equivalent to the data (up to canonical unitary isomorphism) of the object $F(L^2A)\in \lmod{B}$ together with the action of $A$ on $F(L^2A)$ by $B$-linear endomorphisms;
		\item a natural transformation $\eta\colon F \Rightarrow G$ between functors $F, G\colon \lmod{A}\rar \lmod{B}$ is determined by and equivalent to the data of the $A$-equivariant morphism $\eta_{L^2A}:F(L^2A)\rar G(L^2A)$ in $\lmod{B}$.
	\end{enumerate} 
\end{lemma}

\noindent Given two von Neumann algebras $A$ and $B$ as above,
one can recover a functor $F:\lmod{A}\rightarrow \lmod{B}$ from its value on the single object $L^2A$  by the formula ${}_BF(M)={}_BF(L^2A)\boxtimes_A M$. Similarly, given $F, G\colon \lmod{A}\rar \lmod{B}$ as above, one can recover a natural transformation $\eta\colon F \Rightarrow G$ from $\eta_{L^2A}$ by
\begin{equation}\label{eq: nat transf -- generator argument}
\eta_M=\eta_{L^2A}\boxtimes \id_M\colon {}_BF(L^2A)\boxtimes_{A} M\rar {}_BG(L^2A)\boxtimes_{A} M .
\end{equation}

\subsection{The fibre product of von Neumann algebras}
In a typical three-functor formalism, the base change formula is formulated with respect to pullback squares.
Interestingly, in our category of spaces (which is the opposite of the category of commutative von Neumann algebras), our distinguished squares are \emph{not} pullbacks, even though that category does admit pullbacks.
The distinguished squares are instead the \emph{fibre-product squares} of commutative von Neumann algebras.
The notion of fibre-product of two von Neumann algebras over a commutative von Neumann algebra goes back to Sauvageot \cite{Sauvageot1985}, and was later generalised %by Timmermann 
to the case when the base von Neumann algebras is not commutative (\cite[\S2.3]{Enock200} and \cite{Timmermann2012}). 

\begin{definition}\label{def:vnfibreprod}
	Let $\vn{f}\colon C\rightarrow A$ and $\vn{g}\colon C\rightarrow B$ be maps between commutative von Neumann algebras. Then the \emph{fibre product $A \ast_{C} B$} is the von Neumann algebra generated by $A$ and $B$ on $ L^2 A \cofu{C} L^2 B$. 
\end{definition}
The fibre product contains the algebraic tensor product $A\odot_C B$ as a dense subspace.
\begin{remark}
	When $C=L^\infty\ms$ is the von Neumann algebra associated to a standard measure space $\ms$, a map from $C$ to a commutative von Neumann algebra $A$ is equivalent to a measurable bundle $\underline{A} \rar X$ of commutative von Neumann algebras. Given another map $C \rar B$ into a commutative von Neumann algebra $B$, the fibre product $A \ast_{C} B$ agrees with the fibre-wise spatial tensor product of the associated bundles of von Neumann algebras:
	$$
	A \tens{C} B = \int^\oplus_{x\in X} A_x \otimes B_x d\mu(x).
	$$
	For more details, see \cite{Sauvageot1985}.
\end{remark}

Given maps $\vn{f}\colon C \rar A$ and $\vn{g}\colon C \rar B$ of commutative von Neumann algebras, we get corresponding maps $ A\rar A \ast_{C} B$ and $ B \rar A \ast_{C}B$. In that situation, we refer to the resulting commutative square
\begin{equation}\label{eq:tenssquare}
	\begin{tikzcd}
		A\tens{C} B &B \arrow[l,"\bar{\vn{f}}"]\\A\arrow[u,"\bar{\vn{g}}"]&	C \arrow[l,"\vn{f}"]\arrow[u,"\vn{g}"]
	\end{tikzcd}
\end{equation}
as a \emph{fibre product square}. Fibre product squares will be the distinguished squares for our three-functor formalism (cf. Definition~\ref{def:threeffdata}).
The fibre product satisfies no universal property that we are aware of.

In the next section, we will construct several structure maps for the fiber product, and extablish some of their properties.
The most important such map is a certain isomorphism
\begin{equation}\label{eq: Ltwomap}
\Ltwomap : L^2 A \cofu{C} L^2 B\longrightarrow L^2(A \ast_{C} B)
\end{equation}
constructed in Proposition~\ref{prop:tensstandard} below.
The other structure maps that will be relevant for us express the fact that 
the fibre product is unital and associative:
\begin{equation}\label{eq: u and a for fiber product}
A \tens{C} C \cong A
\qquad\text{ and }\qquad 
	(A\tens{E}B)\tens{F}C \cong A\tens{E}(B\tens{F}C).
\end{equation}
The isomorphisms \eqref{eq: u and a for fiber product}
satisfy the usual triangle and pentagon axioms familiar from the theory of monoidal categories \cite[Proposition 9.2.8]{Timmermann2008}.

\begin{remark}
We note that, as observed in \cite[Exa.\,C.2]{Fer:2404.05193}, the fiber product is the operation of composition of $1$-morphisms in a certain 3-category whose objects are commutative von Neumann algebras, and whose $1$-morphisms $C\to D$ are (not necessarily commutative) von Neumann algebras equipped with maps from $C$ and from $D$ into their center.
\end{remark}

\subsection{The isomorphism $L^2 A \cofu{C} L^2 B\rightarrow L^2(A \ast_{C} B)$.
}\label{sec:condexp}
As a first step towards describing the isomorphism \eqref{eq: Ltwomap},
we develop a convenient way of denoting elements of $ L^2 A \cofu{C} L^2 B$ (see Notation~\ref{notation:fusionofLtwos} further down).
Recall that a linear map $\vnmap$ between von Neumann algebras is called \emph{positive} if it sends positive elements to positive elements.\footnote{If $A$ and $B$ are commutative von Neumann algebras, a positive map $\vnmap$ is automatically completely positive, meaning the induced maps $M_n(B)\to M_n(A)$ are also positive.}
Given $\vnmap$ as above, a \emph{non-normalised conditional expectation} is a positive $B$-bilinear normal map $\phi\colon A \rar B$.
We denote the cone of non-normalised conditional expectations by $\Omega(A,B)$. %\bigskip

\begin{terminology}
Conditional expectations (i.e.~non-normalised conditional expectations that satisfy the normalisation condition $\phi(1)=1$) will not play any role in this paper. We therefore allow ourselves to use the term \emph{``conditional expectation''} for elements of $\Omega(A,B)$.
Similarly, we will use the term \emph{``state''} for elements of $\Omega(A,\mathbb{C})$. %\bigskip
\end{terminology}

We denote by $\Hom^\tn{+}_B(L^2B, L^2A) \subset \Hom_B(L^2B, L^2A)$ the cone of positive maps, by which we mean those $B$-linear maps that map $L^2_+B$ to $L^2_+A$.
(Recall that when $B$ is commutative, the left and right actions of $B$ on $L^2B$ agree.)
Since $L^2 B$ is spanned by its positive cone, such a map is entirely determined by its restriction to $L^2_+B$.
The space $\Hom_B(L^2B, L^2A)$ carries a $B$-valued inner product given by
$$
\langle \eta, \nu \rangle_B := \eta^*\nu\in B.
$$

Elements of $L^2_+A$ are in bijection with states on $A$. The bijection is denoted
\begin{equation}\label{eq:posconeposstate}
	\homtocond{(\;\!\cdot \;\!)}:L^2_+A \,\to\, \Omega(A,\mathbb{C}),
\end{equation}
and takes a vector $\xi\in L^2_+A$ to the state $\xi^2$ given by $\homtocond{\xi}(a):=\langle \xi, a \xi\rangle$. The inverse of this map is denoted $\phi\mapsto \sqrt{\phi}$. The above bijection generalises as follows:

\begin{lemma}\label{lem:homtocondexp}
	Let $B\rar A$ be a map of commutative von Neumann algebras. Then there is a bijection
\begin{equation}\label{eq: eta square}
\Hom^\tn{+}_B(L^2B, L^2A) \xrightarrow{\cong} \Omega(A,B)
\end{equation}
denoted $\eta\mapsto \homtocond{\eta}$. It sends $\eta\in \Hom^\tn{+}_B(L^2B, L^2A)$ to the map
$
	\homtocond{\eta}  \colon a \mapsto \langle \eta, a \eta \rangle_B.
$ %	\end{equation*}
\end{lemma}

\begin{proof}
We give the proof in \S\ref{sec:proofoflem}.
\end{proof}

For $\phi\in\Omega(A,B)$, we denote by $\phi\mapsto \sqrt\phi$ the inverse of \eqref{eq: eta square}.
Before proceeding,
we record the statement of
Lemma~\ref{last lemma}, whose proof also appears in \S\ref{sec:proofoflem}:
\begin{equation}\label{last eq}
\sqrt{\phi}^* a \sqrt{\phi} = \phi(a).
\end{equation}

\begin{notation}\label{notation:fusionofLtwos}
Let $A\leftarrow C\rar B$ be homomorphisms between commutative von Neumann algebras.
If $\mu\in\Omega(C,\mathbb C)$ is a state on $C$, and if $\phi\in \Omega(A,C)$ and $\psi\in \Omega(B,C$) are conditional expectations, we write $$\sqrt{\phi}\cofu{\mu}\sqrt{\psi}\in L^2(A)\cofu{C}L^2(B)$$ for the element of corresponding to 
	$
	\sqrt{\phi}\otimes \sqrt{\mu \psi} \in \Hom_C(L^2(C), L^2(A))\otimes_{C} L^2(B).
	$ 
Such elements
span $L^2(A)\cofu{C}L^2(B)$.
\end{notation}
    
With the above notation, we get the following convenient description of the unitor $L^2(A) \cofu{B} L^2B \rar L^2A$, as
    \begin{equation}\label{e:formofunitor}
    \sqrt{\varphi}\underset{\mu}\boxtimes \sqrt{\psi}\mapsto
\sqrt{\psi}\sqrt{\mu\varphi},
    %    b \boxtimes_\mu \sqrt{\phi} \rar b \sqrt{\mu \phi}. 
    \end{equation}
where the $\sqrt{\psi}$ 
in the right hand side refers to the usual square root (as opposed to the inverse of \eqref{eq: eta square})
of the positive element of $B$ corresponding to $\psi:L^2B\to L^2B$.\bigskip

We now address the main goal of this section, which is the construction of the isomorphism \eqref{eq: Ltwomap}:

\begin{proposition}\label{prop:tensstandard}
Let $C\rar A$ and $C\rar B$ be maps of commutative von Neumann algebras. Then the map $\Ltwomap\colon 	L^2(A)\cofu{C}L^2(B) \rar  L^2(A\ast_{C}B)$ defined by
	$$
	\Ltwomap\colon 	\sqrt{\phi} \cofu{\mu}\sqrt{\psi} \mapsto 	\sqrt{\mu (\phi \otimes \psi)}
	$$
	is an isometric isomorphism (i.e.~unitary).
\end{proposition}

\begin{proof}
Write $C$ as $C=\bigoplus C_i$ where each $C_i$ admits a faithful state, and consider the corresponding decompositions $A=\bigoplus A_i$ and $B=\bigoplus B_i$, with maps $C_i \rar A_i$ and $C_i \rar B_i$. We can further decompose $A_i=\bigoplus A_{ij}$ where each $A_{ij}$ admits a faithful conditional expectation $A_{ij} \rar C_i$, and similarly for $B_i$.
This reduces the statement of the proposition to the special case when $C$ admits faithful states, and $A$ and $B$ admit faithful conditional expectations to $C$.\footnote{When the von Neumann algebras $A$, $B$, and $C$ are representable on separable Hilbert spaces, then faithful conditional expectations and states always exist, and there is no need to consider direct sum decompositions.}

Let $\phi:A\to C$ and $\psi:B\to C$ be faithful conditional expectations, and let $\mu:C\to \mathbb C$ be a faithful state, so that
$\sqrt{\phi} \cofu{\mu}\sqrt{\psi}\in L^2(A)\cofu{C}L^2(B)$. Then $\mu \phi$ and $\mu \psi$ are faithful states on $A$ and $B$ respectively, and $\mu(\phi \otimes\psi)$ is a faithful state on $A \ast_{ C} B$.
We first define a map
\begin{align*}
     \Ltwomap_{\phi,\mu,\psi}\colon& L^2(A)\cofu{C}L^2(B) \rar  L^2(A\tens{C}B)\\
    & (a\sqrt{\phi})\cofu{\mu}(b\sqrt{\psi})\mapsto a\otimes b \sqrt{\mu (\phi \otimes \psi)}
\end{align*}
which a priori depends on $\phi$, $\psi$, and $\mu$.
We claim that $\Ltwomap_{\phi,\mu,\psi}$ is a unitary isomorphism, and that it agrees with $\Ltwomap$.

To see that $ \Ltwomap_{\phi,\mu,\psi}$ is isometric, observe that
$$
\big\|\Ltwomap_{\phi,\mu,\psi}\big((a\sqrt{\phi})\cofu{\mu}(b\sqrt{\psi})\big)\big\|^2=\mu (\phi(a^*a)\psi(b^*b)),
$$
and
\begin{align*}
\big\|(a\sqrt{\phi})\cofu{\mu}(b&\sqrt{\psi})\big\|^2 =
\big\langle
a\sqrt{\phi}\otimes b\sqrt{\mu\psi},a\sqrt{\phi}\otimes b\sqrt{\mu\psi}
\big\rangle
\\[-1mm]
&\quad\,\,\,\,\stackrel{\eqref{e:innerprodonfusion}}=
\big\langle
(\sqrt{\phi}^*a^*a\sqrt{\phi}) b\sqrt{\mu\psi}, b\sqrt{\mu\psi}
\big\rangle
\\&\stackrel{\eqref{last eq}}=
\big\langle
(\phi(a^*a) b\sqrt{\mu\psi}, b\sqrt{\mu\psi}
\big\rangle
=
\mu\psi(\phi(a^*a)b^*b)=\mu (\phi(a^*a)\psi(b^*b)),
\end{align*}
where the last equality uses the $C$-linearity of $\psi$. The map $\Ltwomap_{\phi,\mu,\psi}$ is unitary as it induces a bijection between the dense subsets $A\odot_C B$ on both sides.

To see that $ \Ltwomap_{\phi,\mu,\psi}$ agrees with $\Ltwomap$, it suffices to consider the values of both maps on vectors of the form $\sqrt{\phi'} \cofu{\mu'}\sqrt{\psi'}$ where $\sqrt{\phi'}=a\sqrt{\phi}$, $\sqrt{\psi'}=b\sqrt{\psi}$ and $\sqrt{\mu'}=c\sqrt{\mu}$, and $a$, $b$ and $c$ are positive elements of $A$, $B$ and $C$, respectively. We have
$$
\Ltwomap(\sqrt{\phi'} \cofu{\mu'}\sqrt{\psi'})=\sqrt{\mu'(\phi' \otimes \psi')},
$$
while $ \Ltwomap_{\phi,\mu,\psi}$ sends 
$$
\sqrt{\phi'} \cofu{\mu'}\sqrt{\psi'}=a\sqrt{\phi} \cofu{\mu}cb\sqrt{\psi}\quad\text{to}\quad  a\otimes cb \sqrt{\mu(\phi \otimes \psi)}.
$$
As both $\sqrt{\mu'(\phi' \otimes \psi')}$ and $a\otimes cb \sqrt{\mu(\phi \otimes \psi)}$ are positive elements of $L^2(A \ast_{C}B)$, to see that they are equal, it is sufficient to check that the corresponding states on $A \ast_{C}B$ agree, namely:
\[
\big(a\otimes cb \sqrt{\mu(\phi \otimes \psi)}\;\!\big)^2=\mu'(\phi' \otimes \psi').
\]
Indeed, for $x\otimes y \in A \ast_{C}B$, one checks:
    \begin{align*}
    \!\!&\!\!\big(a\otimes cb \sqrt{\mu(\phi \otimes \psi)}\;\!\big)^2(x\otimes y)=
        \mu(\phi \otimes \psi)\big( (x\otimes y)(a\otimes cb)^2\big) \\
    \!\!&\!\!=\mu(\phi(x a^2)\psi(y c^2b^2))
        =\mu(\phi'(x)\psi'(y)c^2)
         =\mu'(\phi'(x)\psi'(y))=
         \big(\mu'(\phi' \otimes \psi')\big)(x\otimes y).\qedhere
    \end{align*}
\end{proof}

\begin{remark}
Proposition~\ref{prop:tensstandard} also holds when $A$ and $B$ are arbitrary von Neumann algebras and $C$ is a commutative von Neumann algebra that maps to their centres. The proof is however significantly more involved. When $C$ is not commutative, the fiber product $A \ast_{C} B$ is still defined (by \cite{Timmermann2012}), but we are not aware of any statement that would be analogous to the one in Proposition~\ref{prop:tensstandard}, and that could hold in that context.
\end{remark}

The unitor \eqref{eq: u and a for fiber product}
of the fibre product is induced by the unitor $L^2A \cofu{C} L^2 C \cong L^2A$ for fusion.
These two unitors are compatible with the map $\Ltwomap$ in the following sense:

\begin{lemma}\label{lem: unitor = unitor}
Let $C\rar A$ be a map of commutative von Neumann algebras, and let $u:A\ast_{C}C\to A$ be the isomorphism in \eqref{eq: u and a for fiber product}.
Then the composite
\begin{equation}\label{eq: unitor = unitor}
L^2(A)\cofu{C}L^2(C) \xrightarrow{\,\,\,\Ltwomap\,\,\,}  L^2(A\tens{C}C)
\xrightarrow{L^2(u)} L^2(A)
\end{equation}
is the unitor for fusion.
\end{lemma}

\begin{proof}
The map $u$ sends $a\otimes x\in A\ast_{C}C$ to $ax\in A$, and therefore pulls back $\mu\psi\varphi:A\to \mathbb{C}$ to 
$\mu(\varphi\otimes \psi):A\ast_{C}C\to \mathbb{C}$
for every $\varphi\in \Omega(A,C)$, $\psi\in\Omega(C,C)=C_+$, and $\mu\in\Omega(C,\mathbb C)$.
The map \eqref{eq: unitor = unitor} thus sends
$$\sqrt{\varphi}\underset{\mu}\boxtimes \sqrt{\psi}\mapsto
\sqrt{\mu(\varphi\otimes \psi)}
\mapsto \sqrt{\mu\psi\varphi} = \sqrt{\psi}\sqrt{\mu\varphi}.
$$
This agrees with the expression \eqref{e:formofunitor} for the unitor of fusion.
\end{proof}

We review the associativity isomorphism of the fiber product of commutative von Neumann algebras (\cite[Proposition 9.2.8]{Timmermann2008} treats this without the commutativity assumption):

\begin{lemma}\label{lem:fibreprodassoc}
	Let 
	\begin{equation}\label{eq: W}
		\begin{tikzcd}[row sep=tiny, column sep=small]
			B_1  	& 	& B_2  	& 	& B_3 \\
			&A_1\arrow[lu]\arrow[ru]& &\arrow[lu]A_2\arrow[ur]&
		\end{tikzcd}
	\end{equation}
    be a diagram of commutative von Neumann algebras. Then there is a canonical isomorphism
	\begin{equation}\label{eq: 3 = 3}
	(B_1\tens{A_1}B_2)\tens{A_2}B_3 \cong B_1\tens{A_1}(B_2\tens{A_2}B_3).
	\end{equation}
\end{lemma}
\begin{proof}
    By definition, the left hand side of \eqref{eq: 3 = 3} is generated by $B_1\ast_{A_1}B_2$ and $B_3$ on $L^2(B_1\ast_{A_1}B_2)\cofu{A_2}L^2(B_3)$.
    By Proposition~\ref{prop:tensstandard},
    this is the same as the algebra generated by 
    $B_1$, $B_2$ and $B_3$ acting on $L^2(B_1)\cofu{A_1}L^2(B_2)\cofu{A_2}L^2(B_3)$. 
By symmetry, this agrees with the right hand side of \eqref{eq: 3 = 3}.
\end{proof}

The associators of the fibre product and fusion are compatible, as we will show in Lemma~\ref{lem:sass}. The first step towards this is:
\begin{lemma}\label{lemma B <-A -> B -> B}
Given commutative von Neumann algebras
$B_1 \leftarrow A \rightarrow B_2 \rightarrow B_3$, the diagram
    \begin{center}
        \begin{tikzcd}[column sep = small]
        L^2B_1 \underset{A}{\boxtimes}L^2 B_3 \arrow[r,"\Ltwomap"]\arrow[d,"\cong"]& L^2(B_1 \underset{A}{\ast} B_3)\arrow[dd,"\cong"]\\
        \big(L^2B_1 \underset{A}{\boxtimes}L^2 B_2\big) \underset{B_2}{\boxtimes}L^2 B_3\arrow[d,"\Ltwomap"]&\\
         L^2(B_1 \underset{A}{\ast} B_2)\underset{B_2}{\boxtimes}L^2 B_3\arrow[r,"\Ltwomap"]&  L^2(B_1 \underset{A}{\ast} B_2 \underset{B_2}{\ast} B_3)
        \end{tikzcd}
    \end{center}
    commutes.
\end{lemma}

\begin{proof}
As before, we can reduce to the case when $A \rar B_2$ admits a faithful conditional expectation $\nu \in \Omega(B_2,A)$. Then there is a dense set of conditional expectations $B_3 \rar A$ that can be written as $\nu \psi$, for $\psi\in \Omega(B_3, B_2)$. This allows us to represent a dense set of elements of $L^2B_1 {\boxtimes}_A L^2 B_3$ as $\sqrt{\phi} \boxtimes_\mu \sqrt{\nu \psi}$, where $\phi\in \Omega(B_1,A)$ and $\mu \in \Omega(A,\mathbb{C})$. The route along the top of the diagram then becomes:
    $
    \sqrt{\phi} \boxtimes_\mu \sqrt{\nu \psi}\mapsto \sqrt{\mu(\phi \otimes \nu \psi)} \mapsto \sqrt{\mu(\phi \otimes \nu (\id \otimes \psi)) } = \sqrt{\mu \nu (\phi \otimes \id \otimes \psi)}$,
    where in the last equality we used that $\nu$ is $A$-linear.
    
    The route along the bottom requires a little more care. The first map down is a combination of a unitor and an associator. Viewing $L^2B_1 \boxtimes_AL^2 B_2 \boxtimes_{B_2}L^2B_3$ as the completion of $
    \Hom_{A}(L^2A,L^2B_1)
    \otimes_A L^2 B_2 \otimes_{B_2}\Hom_{B_2}(L^2B_2,L^2B_3)$, the first downwards map is given by
    $$
    \sqrt{\phi} \boxtimes_\mu \sqrt{\nu \psi}
    \mapsto
    \big(\underbrace{(\sqrt{\mu} \mapsto \sqrt{\mu \phi})\otimes\sqrt{\mu \nu}}_{\textstyle=\sqrt{\phi}\boxtimes_\mu \sqrt{\nu}}\otimes (\sqrt{\mu\nu}\mapsto \sqrt{\mu \nu \psi})\big).
    $$
    The second map down replaces $\sqrt{\phi}\boxtimes_\mu \sqrt{\nu}$ by $\sqrt{\mu(\phi \otimes\nu)}$, which we can rewrite as $\sqrt{\mu\nu(\phi \otimes \id)}$ because $\nu$ is $A$-linear.
    So, overall, the composition of the two downward maps with the bottom map reads:
    \begin{gather*}
    \sqrt{\phi} \boxtimes_\mu \sqrt{\nu \psi}
    \\[-2mm]\rotatebox{-90}{$\mapsto$}\\
    \big(\sqrt{\phi}\boxtimes_\mu \sqrt{\nu}\big)\otimes \big(\sqrt{\mu\nu}\mapsto \sqrt{\mu \nu \psi}\big)
    \\[-2mm]\rotatebox{-90}{$\mapsto$}\\
    \sqrt{\mu(\phi \otimes\nu)}\otimes \big(\sqrt{\mu\nu}\mapsto \sqrt{\mu \nu \psi}\big)
    \\[-2mm]\rotatebox{-90}{$=$}\\
    \sqrt{\phi\otimes \id}\boxtimes_{\mu \nu} \sqrt{\psi}
    \\[-2mm]\rotatebox{-90}{$\mapsto$}\\
    \sqrt{\mu \nu (\phi \otimes \id \otimes \psi)}.\qedhere
    \end{gather*}
\end{proof}

We finish this section by showing that
the associator \eqref{eq: 3 = 3} for the fiber product is compatible with the map $\Ltwomap$ from Proposition~\ref{prop:tensstandard}:

\begin{lemma}\label{lem:sass}
	For $A_1,A_2, B_1,B_2,B_3$ as in \eqref{eq: W}, the following diagram commutes:
    \begin{center}
		\begin{tikzcd}[column sep=small]
			L^2(B_1)\underset{A_1}\boxtimes\big(L^2(B_2)\underset{A_2}\boxtimes L^2(B_3)\big)\arrow[r,"\Ltwomap\,"]\arrow[d,"\cong"] &L^2(B_1)\underset{A_1}\boxtimes	L^2(B_2\tens{A_2}B_3)\arrow[r,"\Ltwomap"]&,L^2(B_1 \tens{A_1}\big(B_2 \tens{A_2}B_3)\big)\arrow[d,"\cong"]\\
            \big(L^2(B_1)\underset{A_1}\boxtimes L^2(B_2)\big)\underset{A_2}\boxtimes L^2(B_3) \arrow[r,"\Ltwomap"]&L^2(B_1\tens{A_1}B_2)\underset{A_2}\boxtimes L^2(B_3)  \arrow[r,"\Ltwomap\,"]& L^2\big((B_1 \tens{A_1}B_2) \tens{A_2}B_3\big) 
		\end{tikzcd}
	\end{center}
\end{lemma}

\begin{proof}
The triangles and quadrilaterals in the following diagram visibly commute:
\[
\hspace{-2cm}
\begin{tikzpicture}[yscale=1.2]
\node (A0) at (-6,2) {$L^2B_1\underset{A_1}\boxtimes\big( L^2B_2\underset{A_2}\boxtimes L^2B_3\big)$};
\node (B0) at (0,2) {$L^2B_1\underset{A_1}\boxtimes	L^2(B_2\tens{A_2}B_3)$};
\node (C0) at (4.5,2) {$L^2\big(B_1 \tens{A_1}(B_2 \tens{A_2}B_3)\big)$};
\node (A1) at (-6,0) {$(L^2
B_1 \underset{A_1}\boxtimes L^2 B_2)\underset{B_2}\boxtimes \big(L^2B_2 \underset{A_2}\boxtimes L^2B_3\big)
$};
\node (B1) at (-.5,0) {$L^2
(B_1 \tens{A_1} B_2)\underset{B_2}\boxtimes L^2(B_2 \tens{A_2} B_3)
$};
\node (C1) at (4.5,0) {$L^2\big(
(B_1 \tens{A_1} B_2)\tens{B_2}(B_2 \tens{A_2} B_3)
\big)$};
\node (A2) at (-6,-2) {$\big(L^2B_1\underset{A_1}\boxtimes L^2B_2\big)\underset{A_2}\boxtimes L^2B_3$};
\node (B2) at (0,-2) {$L^2(B_1\tens{A_1}B_2)\underset{A_2}\boxtimes L^2B_3$};
\node (C2) at (4.5,-2) {$L^2\big((B_1 \tens{A_1}B_2) \tens{A_2}B_3\big)$};
\node (X) at (-3.125,1) {($L^2
B_1 \underset{A_1}\boxtimes L^2 B_2)\underset{B_2}\boxtimes L^2(B_2 \tens{A_2} B_3)
$};
\node (Y) at (-3.125,-1) {$L^2
(B_1 \tens{A_1} B_2)\underset{B_2}\boxtimes \big(L^2B_2 \underset{A_2}\boxtimes L^2B_3\big)
$};
\draw ($(A0.south)+(-2,0)$) to[bend right=50, looseness=1.5] ($(A2.north)+(-2,0)$);
\draw ($(C0.south)+(1.5,0)$) to[bend left=50, looseness=1.5] ($(C2.north)+(1.5,0)$);
\draw ($(A0.east)+(0,.1)$) -- ($(B0.west)+(0,.1)$)
($(B0.east)+(0,.1)$) -- ($(C0.west)+(0,.1)$);
%\draw (A0) -- (B0) -- (C0);
\draw ($(A2.east)+(0,.1)$) -- ($(B2.west)+(0,.1)$)
($(B2.east)+(0,.1)$) -- ($(C2.west)+(0,.1)$);
%\draw (A2) -- (B2) -- (C2);
\draw (X) -- (B1) (C1) -- (C0) ($(B1.east)+(0,.1)$) -- ($(C1.west)+(0,.1)$);
\draw (Y) -- (B1) (C1) -- (C2);
\draw (A0) -- (A1) -- (A2);
\draw (B0) --(X) -- (A1) -- (Y) -- (B2);
\end{tikzpicture}
\]
and two pentagons commute by Lemma~\ref{lemma B <-A -> B -> B}.
\end{proof}

\subsubsection{Proof of Lemma~\ref{lem:homtocondexp}}\label{sec:proofoflem}

If $B$ is separable, it admits a faithful state. More generally, every von Neumann algebra can be written as a direct sum of von Neumann algebras that admit faithful states.
By decomposing $B$ as a direct sum $\bigoplus B_i$ where each $B_i$ admits a faithful state, we get a corresponding decomposition $A=\bigoplus A_i$ such that the given map $B\rar A$ is a direct sum of maps $B_i\rar A_i$.
We may thus assume without loss of generality that our algebra $B$ admits a faithful state $\mu$.

We exhibit an inverse $\sqrt{\cdot}$ to the map \eqref{eq: eta square}. If $\phi \in \Omega(A,B)$ is a conditional expectation,
then $\mu \phi$ is a state on $A$, and we have a left $B$-linear map
    \begin{align}\label{eq: sqrt phi as a map}
	\sqrt{\phi}\,\colon\,	L^2B &\longrightarrow L^2A\\\notag
    \textstyle b\sqrt{\mu}& \textstyle \,\,\mapsto b\sqrt{\mu \phi}
    \end{align}
defined as the closure of the densely defined map
$B\sqrt \mu\to L^2A:b\sqrt{\mu} \mapsto b\sqrt{\mu \phi}$.
This is a bounded positive map, and we will see below that it is independent of the choice of faithful state $\mu:B\to \mathbb C$.

To see that the map \eqref{eq: sqrt phi as a map} is independent of $\mu$,
let us temporarily denote it $\sqrt[\uproot{2}\mu] \phi$ to emphasise its potential dependence on $\mu$.
If $\mu'=b_1\mu$ for some positive element with full support $b_1\in B$, then
\[
\sqrt[\uproot{2}\mu] \phi(b\sqrt \mu')
=\sqrt[\uproot{2}\mu] \phi(b\sqrt{b_1}\sqrt \mu)
=b\sqrt{b_1}\sqrt {\mu\phi}
=b\sqrt{b_1\mu\phi}
=b\sqrt{\mu'\phi}
=\sqrt[\uproot{2}\mu'] \phi(b\sqrt \mu')
\]
showing that $\sqrt[\uproot{2}\mu] \phi=\sqrt[\uproot{2}\mu'] \phi$.
For $\mu'$ an arbitrary faithful state on $B$, there exist full support positive elements $b_1,b_2\in B$ such that $\mu'':=b_1\mu=b_2\mu'$. It follows that $\sqrt[\uproot{2}\mu] \phi=\sqrt[\uproot{1}\mu^{\prime\!\!\;\prime}] \phi=\sqrt[\uproot{1}\mu'] \phi$.

Given $\phi \in \Omega(A,B)$, we claim that $\homtocond{\sqrt{\phi}}=\phi$. By definition,
	$
	\homtocond{\sqrt{\phi}}(a)= \sqrt{\phi}^* a \sqrt{\phi},
	$
    so we need to show that
    \begin{equation}\label{eq: phi(a)=sqrtphi* a sqrtphi}
    \phi(a)=\sqrt{\phi}^* a \sqrt{\phi}.
    \end{equation}
Since
	$
	\langle a \sqrt{\mu \phi},  \sqrt{\phi}(b\sqrt{\mu})\rangle=\langle a \sqrt{\mu \phi}, b \sqrt{\mu \phi}\rangle= \mu\phi(a^*b)=\mu(\phi(a)^*b) = \langle \phi(a)\sqrt{\mu}, b \sqrt{\mu} \rangle,
	$
	the adjoint $\sqrt{\phi}^*$ sends elements of the form $a\sqrt{ \mu\phi}$ to $ \phi(a)\sqrt{\mu}$:
\begin{equation}\label{eq: adj's formula}
\sqrt{\phi}^*(a\sqrt{ \mu\phi})=\phi(a)\sqrt{\mu}.
\end{equation}
We check \eqref{eq: phi(a)=sqrtphi* a sqrtphi} after evaluating on $b\sqrt \mu\in L^2B$:
	$$
    \sqrt{\phi}^* a \sqrt{\phi}\big(b\sqrt\mu\big)=\sqrt{\phi}^*\big(ab \sqrt{\mu \phi}\big)=\phi(ab) \sqrt{\mu}=\phi(a) b\sqrt{\mu}.
	$$
	So $\homtocond{\sqrt{\phi}}=\phi$.
    
    Given $\eta\in \Hom^\tn{+}_B(L^2B, L^2A)$ we now verify that $\sqrt{\eta^2}=\eta$.  The map $\sqrt{\homtocond{\eta}}$ sends $\srmu$ to $\sqrt{\mu \homtocond{\eta}}$. Since the state $\mu \homtocond{\eta}:A\to \mathbb C$ sends $a\in A$ to
	$$
	\mu\homtocond{\eta}(a)=\mu(\eta^* a \eta ) = \langle \sqrt{\mu}, \eta^* a \eta \sqrt{\mu}\rangle =\langle \eta \srmu , a \eta \srmu \rangle= \homtocond{(\eta \srmu)}(a),
	$$
    we have $\sqrt{\mu\homtocond{\eta}}=\eta\sqrt{\mu}$.
    It follows that $\sqrt{\homtocond{\eta}}(\srmu)=\eta\sqrt{\mu}$ for all $\sqrt{\mu}\in L^2B$, from which we get our desired conclusion: $\sqrt{\homtocond{\eta}}=\eta$.
\hfill $\square$
\medskip

We finish this section by recording a consequence of equation \eqref{eq: adj's formula}:

\begin{lemma}\label{last lemma}
For $\phi\in\Omega(A,B)$, we have
\begin{equation*}
\sqrt{\phi}^* a \sqrt{\phi} = \phi(a).
\end{equation*}
\end{lemma}

\begin{proof}
Evaluating the two sides of the equation on $b\sqrt{\mu}$ for some state $\mu$ on $B$, we check:
$\sqrt{\phi}^* a \sqrt{\phi}(b\sqrt{\mu}) 
=
\sqrt{\phi}^*\big(ab\sqrt{\mu\phi}\big)
= \phi(ab)\sqrt{\mu}
= \phi(a)b\sqrt{\mu}$,
where the middle equality holds by \eqref{eq: adj's formula}.
\end{proof}

\section{Three-functor formalism for commutative von Neumann algebras}\label{sec4}
The goal of this section is to set up our three-functor formalism for commutative von Neumann algebras.  We first introduce the three functors (pullback, pushforward, and tensor product, which in our context become induction, restriction, and fusion), and their structure morphisms. After this, we introduce the structure maps expressing compatibility between the three functors. The final part of this section is devoted to proving all the axioms. 

\subsection{The category of geometric objects}
Our three-functor formalism lives on the category $\cvna$ of commutative von Neumann algebras. It is a contravariant three-functor formalism, meanning that it induces a three-functor formalism with the usual variance on the opposite category $\cvna^\op$ (which is a version of the category of measure spaces).

We will take all morphisms to be exceptional, meaning that the exceptional pushforward is defined on all morphisms. The distinguished squares in $\cvna$ are the fibre product squares \eqref{eq:tenssquare}; it follows from Lemma~\ref{lem:fibreprodassoc} (with $A_2=B_2$) that these squares glue.

\subsection{The three functors}\label{sec:the three functors}
We now introduce the functors from Definition~\ref{def:threeffdata}. Because our three-functor formalism is contravariant the role of ``pushforward'' $\push{(-)}$ will be played by a \emph{contravariant} functor
\begin{align*}
\cvna^\op \,&\longrightarrow\, \Wstarcat\\
A \qquad&\,\,\mapsto\,\,\, \lmod{A}\\
(f:B\to A) &\,\,\mapsto (\res{f}:\lmod{A}\to \lmod{B})\hspace{-1.5cm}
\intertext{corresponding to restriction (that views an $A$-module as an $B$-module along $\vn{f}$). The ``pullback'' $\pull{(-)}$ is replaced by a \emph{covariant} functor}
\cvna \,&\longrightarrow \symWcat\\
A \quad&\,\,\mapsto\,\,\,\,\,\,\,(\lmod{A},\boxtimes_A)\\
(f:B\to A) &\,\,\mapsto (\ind{f}:\lmod{B}\to \lmod{A})\hspace{-1.5cm}
\end{align*}
given by induction: 
$\ind{f}(M):= L^2A \cofu{B} M$ (where $L^2A$ is viewed as a $B$-module using $\vn{f}$).
The identitors and compositors for $\res{(-)}$ are identity natural transformations, and those for
$\ind{(-)}$ come from the associators and unitors of fusion.

The symmetric monoidal structure on $\lmod{A}$ is given by fusion over the commutative von Neumann algebra $A$.
To show that $\vnind{f}$ is monoidal, we need to provide isomorphisms
\begin{equation}\label{monoial structure maps}
\vnind{f} (L^2B) \cong L^2A
\qquad\,\,\,\text{and}
\qquad\,\,\,
	\vnind{f}(M\cofu{B}N)\cong \vnind{f}(M)\cofu{A}\vnind{f}(N)
\end{equation}
for all $M,N\in\lmod{B}$.
The first one $\vnind{f} (L^2B)= L^2A_\vn{f}\cofu{B}L^2B\cong L^2A$
comes from the unitor for $\cofu{B}$.
By Lemma~\ref{lem:mainlemma}, the second set of isomorphisms are determined by the special case $N=L^2B$ (or $N=M=L^2B$), in which case they are given by the following composition of (images under $\vnind{f}$ of) unitors
$$
\vnind{f}(M\cofu{B}L^2B)\cong \vnind{f}(M) \cong\vnind{f}(M)\cofu{A}L^2A\cong \vnind{f}(M)\cofu{A}\vnind{f}(L^2B).
$$
The coherences between the monoidal structure maps \eqref{monoial structure maps} and the compositors for $\ind{(-)}$ are all applications of Lemma~\ref{lem:mainlemma}, using the well known coherence diagrams for unitors and associators in the bicategory of von Neumann algebras \cite{Brouwer2003}.

\subsection{Projection and base change}
To give a three-functor formalism we need to specify natural isomorphisms expressing projection and base change (the second and third bullets in Definition~\ref{def:threeffdata}).

Given a map $\vnmap$ of commutative von Neumann algebras and a module $M \in \lmod{A}$, projection is expressed by a natural isomorphism $\vnres{f}(M)\cofu{B} - \cong \vnres{f}(M \cofu{A} \vnind{f} - )$.
When the argument is $L^2B$, we define this isomorphism to be 
\begin{equation}\label{eq:projectiononltwo}
\vnres{f}(M)\cofu{B}L^2B \cong \vnres{f}(M)\cong \vnres{f}(M\cofu{A}L^2A)\cong  \vnres{f}\big(M \cofu{A} \vnind{f} (L^2B) \big).
\end{equation}
We then use Lemma~\ref{lem:mainlemma} to extend this to other values of the argument. Note that all the isomorphisms used here are unitors and associators for fusion, so \eqref{eq:projectiononltwo} is natural in $M$.

Given $A\stackrel{f}{\leftarrow} C \stackrel{g}{\rightarrow} B$, and $\bar g$, $\bar f$ forming a distinguished square as in~\eqref{eq:tenssquare}, base change is a natural isomorphism $\vnind{f}\vnres{g}\cong \vnres{\bar{g}}\vnind{\bar{f}}$.
When the argument is $L^2B$, the base change isomorphism is defined to be the composite
\begin{align*}
\vnind{f} \vnres{g}(L^2B) \cong \vnind{f}(L^2B) &\cong L^2A \cofu{C} L^2B
\\
&\cong L^2(B \tens{C} A)
\cong \vnres{\bar{g}}(L^2(B \tens{C} A)\cofu{B}L^2B)
\cong \vnres{\bar{g}} \vnind{\bar{f}}(L^2B),
\end{align*}
where the third isomorphism is $\Ltwomap$ from Proposition~\ref{prop:tensstandard}.
The above isomorphism is one of $B$-modules so, by  Lemma~\ref{lem:mainlemma}, we can extend it to a definition of the base change isomorphism to all other values of the argument.

\subsection{Coherences}
We now proceed to check all the coherence diagrams for the structure maps constructed in the previous section. We have already argued that the restriction and induction each satisfy the required coherence to define functors $\vna^\op\rar \Wstarcat$, and $\cvna\rar \symWcat$. This only leaves checking the coherence conditions 1)--10) in Definition~\ref{def:threeffdata}. These are the ones that involve either projection or base-change, or both. For all these coherences, it suffices by Lemma~\ref{lem:mainlemma} to check the commutativity of the diagram on generators of the relevant categories. 

\subsubsection{Coherence for projection} 
We start by considering the coherence conditions 1)--4) in Definition~\ref{def:threeffdata}. With the exception of the base-change isomorphism, all structure morphisms in our three-functor formalism are compositions of the unitors and associators for fusion, making the verifications that the relevant diagrams commute very routine.

1) The diagram is
\begin{center}
\begin{tikzcd}
		M\cofu{A}N \arrow[r,"="]\arrow[d,"\cong"] & M\cofu{A}N\arrow[d,"="]\\
		M\cofu{A}(L^2A\cofu{A}N)\arrow[r,"\cong"]& M\cofu{A}N.
\end{tikzcd}
\end{center}
for $M,N\in \lmod{A}$.
The bottom horizontal map is a unitor.
The left vertical map is prescribed to be \eqref{eq:projectiononltwo} when $N=L^2A$, and extended to general $N$ by Lemma~\ref{lem:mainlemma}.
When $N=L^2A$, that left vertical map is easily checked to be a unitor.
By Lemma~\ref{lem:mainlemma}, it is therefore a unitor for all $A$-modules $N$ (this can also be checked by direct inspection of the formula \eqref{eq: nat transf -- generator argument}).
The diagram commutes.

2) Given maps $C\stackrel{g}\to B\stackrel{f}\to A$, and modules $M\in\lmod{A}$, and $N\in\lmod{C}$, the relevant diagram is
\begin{center}
\begin{tikzcd} %[column sep=tiny]
    M\cofu{C} N \arrow[d,"="]\arrow[r,"\cong"]& M\cofu{B}(L^2B\cofu{C}N) \arrow[r,"\cong"]& 
    M\cofu{A}(L^2A\cofu{B}(L^2B\cofu{C}N))\arrow[d,"\cong"]\\
    M\cofu{C} N \arrow[rr,"\cong"]&& M\cofu{A} (L^2A\cofu{C}N)
\end{tikzcd}
\end{center}
The horizontal isomorphisms are defined through Lemma~\ref{lem:mainlemma} (applying \eqref{eq: nat transf -- generator argument} to \eqref{eq:projectiononltwo}), and are compositions of unitors and associators. 
The rightmost isomorphism is a unitor.
The diagram commutes.

3) For $\vn{f}\colon B\rar A$ and $M\in \lmod{A}$, we need to show that the diagram
\begin{center}
	\begin{tikzcd}
		M\cofu{B}L^2B \arrow[r,"\cong"]\arrow[d,"\cong"] & M\arrow[d,"\cong"]\\
		M\cofu{A}(L^2A\cofu{B}L^2B)\arrow[r,"\cong"]& M\cofu{A}L^2A.
	\end{tikzcd}
\end{center}
commutes.
Once again, all the isomorphisms are just unitors and associators.

4) Given $\vn{f}\colon B\rar A$ and $M\in \lmod{A}$, $N,P\in \lmod{B}$, we must show
\[
\begin{tikzcd}
    \vnres{f}M\otimes N \otimes P \arrow[r,"\cong"]\arrow[d,"\cong"]& \vnres{f}\big(M \otimes \ind{f}N \big)\otimes P \arrow[d,"\cong"]\\
    \vnres{f}\big(M \otimes \ind{f}(N\otimes P)\big) \arrow[r,"\cong"]& \vnres{f}\big(M \otimes \ind{f}N \otimes \ind{f}P \big)
\end{tikzcd}
\]
commutes. This diagram can be redrawn like this:
\[
\begin{tikzcd}
    M\boxtimes_B N \boxtimes_B P \arrow[r,"\cong"]\arrow[d,"\cong"]& (M \boxtimes_A (L^2A\boxtimes_B N) )\boxtimes_B P \arrow[d,"\cong"]\\
    M \boxtimes_A (L^2A\boxtimes_B (N\boxtimes_B P)) \arrow[r,"\cong"]& M \boxtimes_A (L^2A\boxtimes_BN) \boxtimes_A(L^2A\boxtimes_BP) 
\end{tikzcd}
\]
All the maps are composites of unitors and associators. The diagram commutes.

\subsubsection{Coherence for base change}
We now proceed to prove the coherences 
5)--8) in Definition~\ref{def:threeffdata}.
Recall that the definition of the base-change isomorphism uses the isomorphism $\Ltwomap$ from Proposition~\ref{prop:tensstandard}. For the diagrams in 5) and 6)
\[
\begin{tikzcd}
    \ind{f} \vnres{\id} \arrow[r,"\cong"]\arrow[dr,"\cong"']&\vnres{\id}\ind{f}  \arrow[d,"\cong"]\\ & \ind{f}
\end{tikzcd}\qquad
\begin{tikzcd}
    \ind{\id} \vnres{f} \arrow[r,"\cong"]\arrow[dr,"\cong"']& \vnres{f}\ind{\id} \arrow[d,"\cong"]\\ & \vnres{f}
\end{tikzcd}.
\]
involving the identitors, note that if $\vn{f}=\id$ in the fibre product square~\eqref{eq:tenssquare}, then $\vn{g}=\vn{\bar{g}}$ and $\vn{\bar{f}}=\id$, and $A\ast_{C}B\cong B$. Similarly, if $\vn{g}=\id$, we get $A\ast_{C}B\cong A$. By Lemma~\ref{lem: unitor = unitor}, the base-change isomorphism in these diagrams can be identified with a unitor. The two triangle diagrams are easily seen to commute.

7) Given distinguished squares
\begin{equation*} %\label{eq: gluing of distinguished squares}
\begin{matrix}A_2\tens{A_1}(A_1\tens{C}B)\,\cong \\[12mm] \end{matrix}\!
\begin{tikzcd}[baseline=0]
    A_2\tens{C}B& \arrow[l,"\bar{f}_2"]A_1\tens{C} B& \arrow[l,"\bar{f}_1"] B \\
    A_2 \arrow[u,"\bar{\bar{g}}"] & \arrow[l,"f_2"] A_1 \arrow[u,"\bar{g}"] & \arrow[l,"f_1"] C \arrow[u,"g"]
\end{tikzcd}\qquad\quad
\vspace{-3mm}
\end{equation*}
we must show that
$$
\begin{tikzcd}%[row sep=small]
	\ind{(\vn{f}_2)}\ind{(\vn{f}_1)}\vnres{g} \arrow[r, "\cong" ] \arrow[d,"\cong"]& \ind{(\vn{f}_2)}\vnres{\bar{g}}\ind{(\vn{\bar{f}}_1)} \arrow[r,"\cong"]
	& \arrow[d,"\cong"] \vnres{\bar{\bar{g}}} \ind{(\vn{\bar{f}}_2)}\ind{(\vn{\bar{f}}_1)}\\
	\ind{(\vn{f}_2\vn{f}_1)}\vnres{g} \arrow[rr,"\cong"] && \vnres{\bar{\bar{g}}} \ind{(\vn{\bar{f}_2}\vn{\bar{f}_1})}
\end{tikzcd}
$$
commutes.
By Lemma~\ref{lem:sass}, it is enough to check that the diagram commutes when evaluated on $L^2B\in\lmod{B}$. Up to some unitors and associators, this is:
\[
\hspace{-1cm}
\begin{tikzpicture}[yscale=.9]
\node (A) at (0,3.5) {$L^2A_2\cofu{A_1}L^2A_1\cofu{C} L^2B$};
\node (B) at (4.5,3.5) {$L^2A_2\cofu{A_1}L^2( A_1\tens{C}B)$};
\node (C) at (2.2,1.5) {$L^2(A_2\tens{A_1}A_1)\cofu{C} L^2B$};
\node (D) at (6,1.5) {$L^2 (A_2\tens{A_1}A_1\tens{C}B)$};
\node (D') at (10,2.2) {$L^2 (A_2\tens{C}B)\underset{A_1\tens{C}B}\boxtimes L^2(A_1\tens{C}B)$};
\node (D'') at (10,3.5) {$L^2 (A_2\tens{A_1}A_1\tens{C}B)\underset{A_1\tens{C}B}\boxtimes L^2(A_1\tens{C}B)$};
\node (E) at (0,0) {$L^2A_2\cofu{C} L^2B$};
\node (F) at (10,0) {$L^2 (A_2\tens{C}B)$};
\draw [->] (A) --node[above]{$\scriptstyle \Ltwomap$} (B);
\draw [->] (A) --node[above]{$\scriptstyle \Ltwomap$} (C);
\draw [->] (A) -- (E);
\draw [->] ($(B.south)+(-.2,.05)$) --node[right]{$\scriptstyle \Ltwomap$} (D);
\draw [->] ($(C.east)+(0,.1)$) --node[above]{$\scriptstyle \Ltwomap$} ($(D.west)+(0,.1)$);
\draw [->] (C) -- (E);
\draw [->] (D) -- (F);
\draw [->] (E) --node[above]{$\scriptstyle \Ltwomap$} (F);
\draw [->] (D') -- (F);
\draw [->] (D) -- +(2,1.6);
\draw [shorten <=-3, ->] (D'') --node[left, pos=.2]{$\scriptscriptstyle \cong$} (D');
\draw [->] (B) -- (D'');
\end{tikzpicture}.
\]
The top left quadrilateral commutes by Lemma~\ref{lem:sass},
the left triangle commutes by Lemma~\ref{lem: unitor = unitor}, and the bottom quadrilateral commutes by naturality of $\Ltwomap$.
The top right triangle commutes by the definition of the base change isomorphism, and the rightmost quadrilateral commutes by naturality of unitors.

8) Given distinguished squares
\[
\hspace{-.2cm}
\begin{tikzcd}[baseline=0]
  \hspace{-2.75cm}(A\tens{C} B_1)\tens{B_1} B_2 \cong A\tens{C} B_2  \arrow[<-, r,"\bar{\bar{f}}"']\arrow[<-, d,"\bar{g}_2"]&\arrow[<-, d,"g_2"]B_2\\ A\tens{C} B_1\arrow[<-, r,"\bar{f}"'] \arrow[<-, d,"\bar{g}_1"] & B_1 \arrow[<-, d,"g_1"]\\ A \arrow[<-, r,"f"']& C
\end{tikzcd}
\]
we must show that
\[
\begin{tikzcd}
\ind{f}\vnres{g_1{}}\vnres{g_2{}}\arrow[r,"\cong"]\arrow[d,"\cong"]& \vnres{\bar{g}_1{}}\ind{\bar{f}}\vnres{g_2{}}\arrow[r,"\cong"] &\vnres{\bar{g}_1{}}\vnres{\bar{g}_2{}} \ind{\bar{\bar{f}}} \arrow[d,"\cong"]\\
\ind{f}\vnres{(g_2g_1)} \arrow[rr,"\cong"] & & \vnres{(\bar{g}_2\bar{g}_1)}\ind{\bar{\bar{f}}}
\end{tikzcd}
\]
commutes.
Invoking Lemma~\ref{lem:mainlemma} once again, it is enough to check the commutativity of the diagram upon evaluation on $L^2B_2\in\lmod{B_2}$. Up to some unitors and associators, this is:
\[
\begin{tikzcd}
L^2A\underset{C}\boxtimes L^2B_2\arrow[r,"\cong"]\arrow[-, dd, double, double distance=2]& L^2(A\tens{C}B_1)\underset{B_1}\boxtimes L^2B_2\arrow[r,"\Ltwomap"] &L^2(
A\tens{C} B_1\tens{B_1} B_2) \arrow[-, dd, "\cong"]\\
&
L^2A\underset{C}\boxtimes L^2B_2\underset{B_1}\boxtimes L^2B_1
\arrow[u, "\Ltwomap"'] \arrow[<-, shorten <=8, shorten >=-7, lu] \arrow[<-, shorten >=1, shorten <=-5,ld]
\\
L^2A\underset{C}\boxtimes L^2B_2 \arrow[rr,"\Ltwomap"] & & L^2(A\tens{C}B_2)
\end{tikzcd}.
\]
The top left triangle commutes by definition  (through equation \eqref{eq: nat transf -- generator argument}) of base-change,

and the pentagon commutes by Lemma~\ref{lemma B <-A -> B -> B}.

\subsubsection{Coherence between projection and base change}
The last two items that need checking in Definition~\ref{def:threeffdata} are 9) and 10).\vspace{-1mm}
The diagrams whose commutativity we must establish are~\eqref{eq:bcandprojcohone} and~\eqref{eq:bcandprojcohtwo}.
Let $A\stackrel{f} \leftarrow C\stackrel{g} \rightarrow B$ and 
$A\stackrel{\bar g} \rightarrow A\ast_{C} B\stackrel{\bar f} \leftarrow B$
be as in~\eqref{eq:tenssquare}.

9) By Lemma~\ref{lem:mainlemma}, it is sufficient to establish~\eqref{eq:bcandprojcohone} for $M=L^2B$ and $N=L^2C$. The diagram is:
\begin{equation*} %\label{eq:vnbcprojdiagone}
\hspace{-.7cm}	
    \begin{tikzcd}[column sep=tiny, row sep=small]
		\vnind{f}\left(\vnres{g}(L^2B)\otimes L^2C\right)\arrow[r,"\cong"]\arrow[d,"\cong"]&\arrow[r,"\cong"]	\vnind{f} \vnres{g}(L^2B)\otimes \vnind{f}(L^2C)&\arrow[d,"\cong"]	\vnres{\bar{g}}\vnind{\bar{f}}(L^2B)\otimes \vnind{f}(L^2C)\\
		\vnind{f} \vnres{g}\left(L^2B\otimes \vnind{g}(L^2C)\right)\arrow[d,"\cong"]& &\arrow[d,"\cong"]\vnres{\bar{g}}\left(\vnind{\bar{f}}(L^2B)\otimes \vnind{\bar{g}} \vnind{f}(L^2C)\right)\\ 
	\vnres{\bar g}	\vnind{\bar{f}}\left(L^2B\otimes \vnind{g}(L^2C)\right)\arrow[r,"\cong"]& \vnres{\bar{g}}\left(\vnind{\bar{f}}(L^2B)\otimes \vnind{\bar{f}} \vnind{g}(L^2C)\right)\arrow[r,"\cong"]&\vnres{\bar{g}}\left(\vnind{\bar{f}}(L^2B)\otimes \vnind{(g\bar{f})} (L^2C)\right)
	\end{tikzcd}
\end{equation*}
which we expand to:

\[
\hspace{-2cm}
\begin{tikzpicture}
\node (A1) at (-.7,5) {$L^2A\underset C\boxtimes( L^2B \underset C\boxtimes
 L^2C)$};
\node (B1) at (4,5) {$(L^2A \underset C\boxtimes
 L^2B)\underset A\boxtimes(L^2A \underset C\boxtimes
 L^2C)$};
\node (C1) at (10,5) {$\big(L^2(A\tens{C}B)\underset B\boxtimes L^2B\big)\underset A\boxtimes(L^2A \underset C\boxtimes
 L^2C)$};
\node (A2) at (-.7,3.1) {$L^2A \underset C\boxtimes(
 L^2B\underset B\boxtimes(L^2B \underset C\boxtimes
 L^2C))$};
\node (B2) at (2.1,3.9) {$L^2A\underset C\boxtimes L^2B$};
\node (B2') at (4,3) {$L^2(A\tens{C}B)$};
\node (A3) at (-.7,1.5) {$L^2(A\tens{C}B)\underset B\boxtimes (L^2B \underset B\boxtimes (L^2B \underset C\boxtimes L^2C))$};
\node (B3) at (4,.5) {$\big(
L^2(A\tens{C}B)\underset B\boxtimes L^2B \big)\underset{A\tens{C}B}\boxtimes\big(
L^2(A\tens{C}B)\underset B\boxtimes (L^2B\underset C\boxtimes L^2C)
\big)$};
\node (C2) at (10,2.9) {$\big(
L^2(A\tens{C}B)\underset B\boxtimes L^2B \big)\underset{A\tens{C}B}\boxtimes\big(
L^2(A\tens{C}B)\underset A\boxtimes (L^2A\underset C\boxtimes L^2C)
\big)$};
\node[inner sep=1] (C3) at (10,1.5) {$\big(
L^2(A\tens{C}B)\underset B\boxtimes L^2B \big)\underset{A\tens{C}B}\boxtimes\big(
L^2(A\tens{C}B)\underset C\boxtimes L^2C
\big)$};
\draw (A2) -- (A1) -- (B1) --node[above]{$\scriptstyle \Ltwomap$} (C1) -- (C2);
\draw (A2) -- (B2) (B2') -- (C2);
\draw (A2) --node[right]{$\scriptstyle \Ltwomap$} (A3) -- (B3) -- (C3) (C2) -- (B2') -- (B3);
\draw (A1) -- (B2) (B2) -- (B1);
\draw (A3) -- (B2') (B2') -- (C1);
\draw[shorten <=-5] (B2) --node[above]{$\scriptstyle \Ltwomap$} (B2');
\draw[shorten >=3] (C2) -- (C3);
\draw[shorten >=7] (B2') -- (C3);
\end{tikzpicture}
\vspace{-2mm}
\]
The triangles in this diagram are composites of unitors and associators, hence commute. The remaining two squares involve $\Ltwomap$, and commute by naturality of the unitors and associators.

10) Once again by Lemma~\ref{lem:mainlemma}, it suffices to establish~\eqref{eq:bcandprojcohtwo} for $M=L^2A$ and $N=L^2C$. The diagram is:

\begin{equation*} %\label{eq:vnbcprojdiagtwo}
\hspace{-.7cm}	
    \begin{tikzcd}[column sep=small,row sep= small]
		\vnres{f}(L^2A)\otimes \vnres{g}(L^2B)\arrow[r,"\cong"]\arrow[d,"\cong"]& 	\vnres{g}\left(\vnind{g} \vnres{f}(L^2A)\otimes L^2B\right)\arrow[r,"\cong"]& 	\vnres{g}\left(\vnres{\bar{f}}\vnind{\bar{g}}(L^2A)\otimes L^2B\right)\arrow[d,"\cong"]\\
		\vnres{f} \left(L^2A\otimes \vnind{f} \vnres{g}(L^2B)\right)\arrow[d,"\cong"]& &	\vnres{g}\vnres{\bar{f}}\left(\vnind{\bar{g}}(L^2A)\otimes  \vnind{\bar{f}}(L^2B)\right)\arrow[d,"\cong"]\\ 
		\left(L^2A\otimes \vnres{\bar{g}} \vnind{\bar{f}}(L^2B)\right)\arrow[r,"\cong"]&\vnres{f}\vnres{\bar{g}}\left(\vnind{\bar{g}}(L^2A)\otimes \vnind{\bar{f}}(L^2B)\right)\arrow[r,"\cong"]&\vnres{(f\bar{g})}\left(\vnind{\bar{g}}(L^2A)\otimes \vnind{\bar{f}}(L^2B)\right).
	\end{tikzcd}	
\end{equation*}
The three entries in the lower right corner are all equal (and connected by identity arrow), so the diagram simplifies to: 
\[
\hspace{-.9cm}
\begin{tikzpicture}
\node (A1) at (-.7,4.3) {$L^2A\underset C\boxtimes L^2B$};
\node (A1') at (.5,5) {$L^2B\underset C\boxtimes L^2A$};
\node (B1) at (4.2,5) {$(L^2B\underset C\boxtimes L^2A)\underset B\boxtimes L^2B$};
\node (C1) at (9,5) {$\big(L^2(A\tens{C}B)\underset A\boxtimes L^2A\big)\underset B\boxtimes L^2B$};
\node (A2) at (-.7,3) {$L^2A\underset A\boxtimes(L^2A\underset C\boxtimes L^2B)$};
\node (B2) at (2.8,3.8) {$L^2A\underset C\boxtimes L^2B$};
\node (B2') at (5,2.9) {$L^2(A\tens{C}B)$};
\node (A3) at (-.7,1.5) {$L^2A\underset A\boxtimes\big(L^2(A\tens{C}B)\underset B\boxtimes L^2B\big)$};
\node (C3) at (9,1.5) {$\big(
L^2(A\tens{C}B)\underset A\boxtimes L^2A \big)\underset{A\tens{C}B}\boxtimes\big(
L^2(A\tens{C}B)\underset B\boxtimes L^2B
\big)$};
\draw (A2) -- (A1) (A1') -- (B1) --node[above]{$\scriptstyle \Ltwomap$} (C1) -- (C3);
\draw (A2) -- (B2);
\draw[shorten <=-5, shorten >=3] (A1) -- (A1');
\draw[shorten <=-5] (B2) --node[above]{$\scriptstyle \Ltwomap$} (B2');
\draw (A2) --node[right]{$\scriptstyle \Ltwomap$} (A3) -- (C3);
\draw (A1) -- (B2) -- (A1') (B2) -- (B1);
\draw (A3) -- (B2') -- (C3) (B2') -- (C1);
\end{tikzpicture}
\]
The triangles consist of unitors and associators (and a symmetry for the top left triangle), hence commute. And the two squares involving $\Ltwomap$ commute by the naturality of unitors and associators.
\medskip

This finishes the proof that our three-functor formalism satisfies all the coherences in~Definition~\ref{def:threeffdata}.

\subsection{Duality and unitarity}
As mentioned in Section~\ref{sec:Variants and enhancements of three-functor formalisms}, our three-functor formalism for commutative von Neumann algebras is involutive (Definition~\ref{def:inv3ff}). This means in particular that each category $\lmod{A}$, for $A$ a commutative von Neumann algebra, comes equipped with a contravariant involutive functor
\[
\mathbb D_A:\lmod{A}\to \lmod{A}.
\]
The duality functor $\mathbb D_A$ sends an $A$-module $M$ to its complex conjugate $\overline{M}$ (with $A$-module structure given by $a\bar \xi:=\overline{a^*\xi}$), and sends an $A$-linear map $h:M\to N$ to its adjoint $h^*:N\to M$ viewed as a map $\overline{N}\to \overline{M}$. The assignment 
\[
(h:M\to N)\mapsto (\mathbb D_Ah:\overline{N}\to \overline{M})
\]
is $\mathbb C$-linear and contravariant.
It remains to argue that $\mathbb D_{(-)}$ commutes with $\vnind f$ and $\vnres f$ (as formulated in \eqref{eq: map xi}), and to check the coherences listed in Definition~\ref{def:inv3ff}. All the coherences follow trivially from the fact that every commutative von Neumann algebra is canonically the complexification of a real von Neumann algebra, and that the three functors $\vnind f$, $\vnres f$, $-\boxtimes_A-$ are complexifications of functors defined over~$\mathbb R$.

The categories $\lmod{A}$ are moreover dagger categories (with the dagger operation provided by the adjoint of a bounded linear map).
The functors $\vnind{f}$, $\vnres{f}$, $-\boxtimes_A-$, and $\mathbb D_A$ are dagger functors,
and all the natural isomorphisms ever mentioned are unitary.
So our three-functor formalism is a bi-involutive three-functor formalism.

Even more, the categories $\lmod{A}$ are $W^*$-categories
and the functors $(\vnind{f},\vnres{f})$ and $(M\boxtimes_A -,\overline M \boxtimes_A -)$ form adjoint pairs in the sense of \cite[Def~4.14]{Henriques2024} (by \cite[Lem~4.16]{Henriques2024}), making our three-functor formalism into a \emph{unitary three-functor formalism}.

\section{An application: Fell absorption for measure groupoids}\label{sec:fellabs}

In this last section, we illustrate the usefulness of our three-functor formalism with a concrete application. We prove that categories of representations of measure groupoids satisfy the Fell absorption principle: the regular representation tensorially absorbs every faithful unitary representation (see Corollary~\ref{last cor} for a precise statement).

Hereafter, all measure spaces are assumed to be standard, and all Hilbert spaces are assumed to be separable.
As explained in \S\ref{sec:threeffmeasure}, 
the three-functor formalism constructed in \S\ref{sec4} can be interpreted as a three-functor formalism on the category of measure spaces and measurable maps.

Recall that a \emph{groupoid} is a small category in which all morphisms are invertible.
Given a groupoid $G$, we write $G_0$ for its set of objects, $G_1$ for its set of morphisms, and $G_2:=G_1\times_{G_0}G_1$ for its set of composable pairs of morphisms. 
Let us write $s,t:G_1\to G_0$ for the maps which send a morphism to its source and target, respectively.
And let us write $p,m,q:G_2\to G_1$ for the projection onto the first factor, the composition of morphisms, and the projection onto the second factor.
For $i\in\{0,1,2\}$, let us also write $\pi_i:G_i\to *$ for the unique map to the point. Altogether, we have maps:
\begin{equation}\label{eq: measure groupoid : all the maps}
\begin{tikzpicture}[baseline=-20]

\node (A) at (0,0) {$G_2$};

\node (B) at (2,0) {$G_1$};

\node (C) at (4,0) {$G_0$};

\node (D) at (6,0) {$*$};

\draw [->] (A.east) + (0,.3) --node[above, yshift=1]{$p$} ($(B.west) + (0,.3)$);

\draw [->] (A) --node[above, yshift=-2]{$m$} (B);

\draw [->] (A.east) + (0,-.3) --node[below, yshift=1]{$q$} ($(B.west) + (0,-.3)$);

\draw [->] (B.east) + (0,.15) --node[above]{$s$} ($(C.west) + (0,.15)$);

\draw [->] (B.east) + (0,-.15) --node[below]{$t$} ($(C.west) + (0,-.15)$);

\draw [->] (C) --node[above]{$\pi_0$} (D);

\draw [->] (B.south)+(.15,-.1) to[bend right=30]node[below]{$\pi_1$} ($(D.south)+(-.2,.08)$);

\draw [->] (A.south)+(.1,-.2) to[bend right=45]node[below]{$\pi_2$} ($(D.south)+(-.07,0)$);

\end{tikzpicture}
\end{equation}

A \emph{measure groupoid} is a groupoid such that both $G_0$ and $G_1$ are standard measure spaces, and such that all the maps listed in \eqref{eq: measure groupoid : all the maps} are measurable (see \cite{anantharaman1956amenable,Mackey,RAMSAY1971253,Hahn1978a,Hahn1978}.

\begin{remark}
The map $\iota:G_0\to G_1$ which sends an object of $G$ to its identity morphism is not required to be measurable. In many examples of interest, the image of $\iota$ has measure zero, hence does not induce a map $L^\infty(G_1)\to L^\infty(G_0)$. 
\end{remark}

We recall the definition of an action of a measure groupoid $G$ on a measurable bundle of Hilbert spaces over $G_0$:

\begin{definition}
Let $H$ be measurable bundle over $G_0$.
A unitary representation of $G$ on $H$ is a unitary isomorphism
\[
\alpha\,:\,\,s^\smsq H\stackrel{\alpha}\longrightarrow
t^\smsq H,
\]
that satisfies $m^\smsq \alpha = q^\smsq \alpha \circ p^\smsq \alpha$.
\end{definition}

If we specialise the above definition to the case when $H\cong \pi_0^\smsq V$ for some Hilbert space $V$, 
then the relevant data becomes that of an isomorphism
\[
\alpha\,:\,\,s^\smsq\pi_0^\smsq V=\pi_1^\smsq V\stackrel{\alpha}\longrightarrow\pi_1^\smsq V=t^\smsq\pi_0^\smsq V
\]
satisfying once again $m^\smsq \alpha = q^\smsq \alpha \circ p^\smsq \alpha$.

The \emph{regular representation} has $t_\smsq \mathbb{I}$ as its underlying bundle, and action:
\[
\lambda\,:\,\,s^\smsq t_\smsq \mathbb{I}\to q_\smsq p^\smsq \mathbb{I}=q_\smsq \mathbb{I}=q_\smsq m^\smsq \mathbb{I}\to t^\smsq t_\smsq \mathbb{I},
\]
where we have applied base change twice, using the following two fibre product squares:
\begin{equation}\label{eq: 2SQ}
\begin{tikzcd}
    G_2 \arrow[r, "p"]\arrow[d,"q"] & G_1 \arrow[d, "t"] \\
    G_1 \arrow[r,"s"] & G_0
\end{tikzcd}
\quad \text{and} \quad
\begin{tikzcd}
    G_2 \arrow[r, "m"]\arrow[d,"q"] & G_1 \arrow[d, "t"] \\
    G_1 \arrow[r,"t"] & G_0
\end{tikzcd}
\end{equation}
(the associated diagrams of commutative von Neumann algebras are fibre product squares in the sense of Definition~\ref{def:vnfibreprod}).

Denote by $(\pi_0^\smsq V)_{\mathrm{triv}}$ the Hilbert bundle $\pi_0^\smsq V$ equipped with the trivial action (where $\alpha$ is the identity). Fell absorption is the statement that the regular representation tensorially absorbs any representation of the form $(\pi_0^\smsq V,\alpha:\pi_1^\smsq V\to\pi_1^\smsq V)$, up to a multiplicity given by $\dim(V)$:

\begin{theorem}\label{thm:fellabs}
Let $G$ be a measure groupoid, let $V$ be a Hilbert space, and let $\alpha$ be an action of $G$ on $\pi_0^\smsq V$.
Then the representations $(\pi_0^\smsq V)_{\mathrm{triv}}\otimes t_\smsq \mathbb{I}$ and $\pi_0^\smsq V\otimes t_\smsq \mathbb{I}$ are isomorphic,
with isomorphism $u$ provided by:
\[
u\,:\,\,\pi_0^\smsq V\otimes t_\smsq \mathbb{I}\to t_\smsq (t^\smsq \pi_0^\smsq V\otimes \mathbb{I})=t_\smsq \pi_1^\smsq V\stackrel{t_\smsq \alpha}\longrightarrow t_\smsq \pi_1^\smsq V=t_\smsq (t^\smsq\pi_0^\smsq V\otimes \mathbb{I})\to \pi_0^\smsq V\otimes t_\smsq \mathbb{I}.
\]
(The first and last isomorphisms in the definition of $u$ are given by projection.)
\end{theorem}

\begin{proof}
The map $u$ is clearly an isomorphism.
We need to show that it is also a morphism of $G$-representations, ie. that it intertwines the action of the groupoid. This is the content of the following diagram:

\[
\hspace{-3.2cm}\begin{tikzpicture}

\node (A1) at (0,9) {$s^\smsq (\pi_0^\smsq V\otimes t_\smsq\mathbb{I})$};

\node (B1) at (0,7.5) {$s^\smsq \pi_0^\smsq V\otimes s^\smsq t_\smsq\mathbb{I}$};

\node (C1) at (0,6) {$\pi_1^\smsq V\otimes q_\smsq p^\smsq\mathbb{I}$};

\node (D1) at (0,4.5) {$\pi_1^\smsq V\otimes q_\smsq\mathbb{I}$};

\node (E1) at (0,3) {$\pi_1^\smsq V\otimes q_\smsq m^\smsq\mathbb{I}$};

\node (F1) at (0,1.5) {$t^\smsq \pi_0^\smsq V\otimes t^\smsq t_\smsq\mathbb{I}$};

\node (G1) at (0,0) {$t^\smsq (\pi_0^\smsq V\otimes t_\smsq\mathbb{I})$};

\node (A2) at (3,9) {$s^\smsq t_\smsq (t^\smsq \pi_0^\smsq V\otimes\mathbb{I})$};

\node (B2) at (3,8) {$q_\smsq p^\smsq (t^\smsq \pi_0^\smsq V\otimes\mathbb{I})$};

\node (B'2) at (3,7) {$q_\smsq (p^\smsq \pi_1^\smsq V\otimes p^\smsq\mathbb{I})$};

\node (C2) at (3,6) {$q_\smsq (q^\smsq \pi_1^\smsq V\otimes p^\smsq\mathbb{I})$};

\node (D2) at (3,4.5) {$q_\smsq (q^\smsq \pi_1^\smsq V\otimes\mathbb{I})$};

\node (E2) at (3,3) {$q_\smsq (q^\smsq \pi_1^\smsq V\otimes m^\smsq\mathbb{I})$};

\node (F2) at (3,2) {$q_\smsq (m^\smsq \pi_1^\smsq V\otimes m^\smsq\mathbb{I})$};

\node (F'2) at (3,1) {$q_\smsq m^\smsq (t^\smsq \pi_0^\smsq V\otimes\mathbb{I})$};

\node (G2) at (3,0) {$t^\smsq t_\smsq (t^\smsq \pi_0^\smsq V\otimes\mathbb{I})$};

\node (A3) at (6,9) {$s^\smsq t_\smsq \pi_1^\smsq V$};

\node (B3) at (6,7.6) {$q_\smsq p^\smsq \pi_1^\smsq V$};

\node (D3) at (6,4.5) {$q_\smsq \pi_2^\smsq V$};

\node (F3) at (6,1.4) {$q_\smsq m^\smsq \pi_1^\smsq V$};

\node (G3) at (6,0) {$t^\smsq t_\smsq \pi_1^\smsq V$};

\node (A4) at (9,9) {$s^\smsq t_\smsq \pi_1^\smsq V$};

\node (B4) at (9,7.6) {$q_\smsq p^\smsq \pi_1^\smsq V$};

\node (C4) at (9,6.5) {$q_\smsq \pi_2^\smsq V$};

\node (D4) at (9,5.2) {$q_\smsq q^\smsq \pi_1^\smsq V$};

\node (D'4) at (9,3.8) {$q_\smsq q^\smsq \pi_1^\smsq V$};

\node (E4) at (9,2.5) {$q_\smsq \pi_2^\smsq V$};

\node (F4) at (9,1.4) {$q_\smsq m^\smsq \pi_1^\smsq V$};

\node (G4) at (9,0)  {$t^\smsq t_\smsq \pi_1^\smsq V$};

\node (A5) at (12,9) {$s^\smsq t_\smsq (t^\smsq \pi_0^\smsq V\otimes\mathbb{I})$};

\node (B5) at (12,8.1) {$q_\smsq p^\smsq (t^\smsq \pi_0^\smsq V\otimes\mathbb{I})$};

\node (B'5) at (12,7.2) {$q_\smsq (p^\smsq \pi_1^\smsq V\otimes p^\smsq\mathbb{I})$};

\node (C5) at (12,6.3) {$q_\smsq (q^\smsq \pi_1^\smsq V\otimes p^\smsq\mathbb{I})$};

\node (D5) at (12,5.2)  {$q_\smsq (q^\smsq \pi_1^\smsq V\otimes\mathbb{I})$};

\node (D'5) at (12,3.8)  {$q_\smsq (q^\smsq \pi_1^\smsq V\otimes\mathbb{I})$};

\node (E5) at (12,2.7) {$q_\smsq (q^\smsq \pi_1^\smsq V\otimes m^\smsq\mathbb{I})$};

\node (F5) at (12,1.8) {$q_\smsq (m^\smsq \pi_1^\smsq V\otimes m^\smsq\mathbb{I})$};

\node (F'5) at (12,.9) {$q_\smsq m^\smsq (t^\smsq \pi_0^\smsq V\otimes\mathbb{I})$};

\node (G5) at (12,0) {$t^\smsq t_\smsq (t^\smsq \pi_0^\smsq V\otimes\mathbb{I})$};

\node (A6) at (15,9) {$s^\smsq (\pi_0^\smsq V\otimes t_\smsq\mathbb{I})$};

\node (B6) at (15,7.6) {$s^\smsq \pi_0^\smsq V\otimes s^\smsq t_\smsq\mathbb{I}$};

\node (C6) at (15,6.3) {$\pi_1^\smsq V\otimes q_\smsq p^\smsq\mathbb{I}$};

\node (D6) at (15,5.2) {$\pi_1^\smsq V\otimes q_\smsq\mathbb{I}$};

\node (D'6) at (15,3.8)  {$\pi_1^\smsq V\otimes q_\smsq\mathbb{I}$};

\node (E6) at (15,2.7) {$\pi_1^\smsq V\otimes q_\smsq m^\smsq\mathbb{I}$};

\node (F6) at (15,1.4) {$t^\smsq \pi_0^\smsq V\otimes t^\smsq t_\smsq\mathbb{I}$};

\node (G6) at (15,0) {$t^\smsq (\pi_0^\smsq V\otimes t_\smsq\mathbb{I})$};

\draw [->] (A1) -- (A2) -- (A3) --node[above]{$\scriptstyle s^\smsq t_\smsq \alpha$} (A4);\draw(A4) -- (A5) -- (A6);

\draw [->] (B2.east) + (0,-.1) -- ($(B3.west) + (0,.1)$) (B3) --node[above]{$\scriptstyle q_\smsq p^\smsq \alpha$} (B4); \draw (B4.east) + (0,.1) -- ($(B5.west) + (0,-.15)$);

\draw (B2.south) + (1.2,0) -- (D3) (B'2.south) + (1,0) -- (D3) (C2.south) + (.8,0) -- (D3);

\draw (C1) -- (C2) (C5) -- (C6);

\draw (B5.south) + (-1.2,0) -- (C4) (B'5.west) + (0,.-.1) -- (C4) (C5.west) -- (C4) (D5.north) + (-1.1,-.05) -- (C4);

\draw (D1) -- (D2) -- (D3) (D4) -- (D5) -- (D6) (D'4) -- (D'5) -- (D'6);

\draw (D'5.south) + (-1.1,.05) -- (E4) (E5.west) -- (E4) (F5.west) + (0,.1) -- (E4) (F'5.north) + (-1.2,0) -- (E4) ;

\draw (E1) -- (E2) (E5) -- (E6);

\draw (F2.north) + (1,0) -- (D3) (F'2.north) + (1.2,0) -- (D3) (E2.north) + (.8,0) -- (D3);

\draw [->]  (F'2.east) + (0,.1) -- ($(F3.west) + (0,-.1)$) (F3) --node[above]{$\scriptstyle q_\smsq m^\smsq \alpha$} (F4); \draw (F4.east) + (0,-.1) -- ($(F'5.west) + (0,.15)$);

\draw [->] (G1) -- (G2) -- (G3) --node[above]{$\scriptstyle t^\smsq t_\smsq \alpha$} (G4);\draw(G4) -- (G5) -- (G6);

\draw (A1) -- (B1) -- (C1) -- (D1) -- (E1) -- (F1) -- (G1);

\draw (A2) -- (B2) -- (B'2) -- (C2) -- (D2) -- (E2) -- (F2) -- (F'2) -- (G2);

\draw (A3) -- (B3) -- (D3) -- (F3) -- (G3);

\draw[->] (A4) -- (B4) -- (C4) -- (D4) --node[right]{$\scriptstyle q_\smsq q^\smsq \alpha$} (D'4);\draw (D'4) -- (E4) -- (F4) -- (G4);

\draw[->] (A5) -- (B5) -- (B'5) -- (C5) -- (D5) --node[right]{$\scriptstyle q_\smsq q^\smsq (\alpha\otimes \id)$} (D'5);\draw (D'5) -- (E5) -- (F5) -- (F'5) -- (G5);

\draw[->] (A6) -- (B6) -- (C6) -- (D6) --node[right]{$\scriptstyle \alpha\otimes \mathrm{id}$} (D'6);\draw (D'6) -- (E6) -- (F6) -- (G6);

\draw [->] (A1) -- +(1.15,1) --node[above]{$\scriptstyle s^\smsq (u)$} ($(A6)+(-1.15,1)$) -- (A6);

\draw [->] (G1) -- +(1.15,-1) --node[below]{$\scriptstyle t^\smsq (u)$} ($(G6)+(-1.15,-1)$) -- (G6);

\draw [->] ($(B1.south)+(-.5,0)$) -- +(-.8,-.6) --node[left]{$\scriptstyle \mathrm{id}\otimes \lambda$} ($(F1.north)+(-1.3,.6)$) -- ($(F1.north)+(-.5,0)$);

\draw [->] ($(B6.south)+(.5,0)$) -- +(.8,-.6) --node[right]{$\scriptstyle\alpha \otimes \lambda$} ($(F6.north)+(1.3,.6)$) -- ($(F6.north)+(.5,0)$);

\end{tikzpicture}
\]
The middle rectangle commutes because $\alpha$ is an action.
The four outer rectangle are instances of item 9)
of Definition~\ref{def:threeffdata}, applied to the fibre product squares~\eqref{eq: 2SQ}.
\end{proof}

\begin{corollary}\label{last cor}
Let $G$ be a measure groupoid.
Let $H$ be a representation of $G$ on with $L^\infty(G_0)$ acts faithfully, and let $\Omega := \bigoplus^\infty \push{t}\mathbb{I}$ be the infinite direct sum of copies of the regular representation. Then $H \otimes \Omega\cong \Omega$.
\end{corollary}
\begin{proof}
    Decompose $G_0$ as $\bigsqcup_{i=1}^\infty G^i_0$, such that each $H|_{G^i_0}$ has rank $i$. This induces a disjoint union decomposition of the whole groupoid as $G=\bigsqcup_{i=1}^\infty G^i$. 
    On each $G^i$, the representation $H$ restricts to one whose underlying $L^\infty(G^i_0)$-module is trivial of rank $i$.
    By Theorem~\ref{thm:fellabs}, we have $H|_{G^i} \otimes \Omega|_{G^i}\cong \Omega|_{G^i}$.
    Assembling all the pieces, we get the desired isomorphism $H \otimes \Omega\cong \Omega$.
\end{proof}

%\bibliographystyle{alpha}
%\bibliography{lib}

\end{document}